\documentclass[10pt, A4, leqno, showkeys]{amsart} 
\date{\today}
\usepackage[active]{srcltx}
\usepackage[]{graphics}
\usepackage{here}
\usepackage{mathtools}
\usepackage[mathscr]{eucal}
\usepackage{mathrsfs}
\usepackage{amscd}
\usepackage{amsfonts}
\usepackage{amsmath,amsthm,amssymb}
\usepackage{latexsym}
\usepackage{comment}
\usepackage{xcolor}
\usepackage[colorlinks=false,linkbordercolor=green,citebordercolor=red]{hyperref}
\usepackage{setspace}
\usepackage{enumitem}

\usepackage{caption} 
\DeclareCaptionType{listcase}[List]

\makeatletter
\AtBeginEnvironment{proof}{%
  \@afterindenttrue            
  \setlength{\parindent}{1em}  
  \setlength{\listparindent}{1em}
}
\makeatother

\numberwithin{equation}{section}
\usepackage{bm}

\usepackage{amsthm}
\theoremstyle{plain}
\newtheorem{theorem}{Theorem}[section]
\newtheorem{proposition}[theorem]{Proposition}
\newtheorem{lemma}[theorem]{Lemma}
\newtheorem{corollary}[theorem]{Corollary}
\newtheorem{fact}[theorem]{Fact}

\newtheorem{problem}[theorem]{Problem}
\newtheorem{question}[theorem]{Question}

\theoremstyle{definition}
\newtheorem{definition}[theorem]{Definition}
\newtheorem{example}[theorem]{Example}
\newtheorem{remark}[theorem]{Remark}
\newtheorem*{acknowledgements}{Acknowledgements}

\newcommand{\R}{\boldsymbol{R}}

\newcommand{\Z}{\boldsymbol{Z}}
\newcommand{\C}{\boldsymbol{C}}
\newcommand{\rt}{\sqrt}
\newcommand{\fr}{\frac}

\newcommand\Res{\operatorname{Res}}
\renewcommand\Re{\operatorname{Re}}
\renewcommand\Im{\operatorname{Im}}
\renewcommand\bar{\overline}
\renewcommand\phi{\varphi}
\renewcommand\Sigma{\varSigma}
\renewcommand\Phi{\varPhi}
\newcommand\ord{\operatorname{ord}}

\renewcommand\tilde{\widetilde}
\renewcommand\epsilon{\varepsilon}
\newcommand{\BSigma}{\bbar{\Sigma}}

\newcommand{\bbar}{\overline}

\newcommand{\al}{\alpha}
\newcommand{\bt}{\beta}
\newcommand{\gm}{\gamma}

\newcommand{\Chat}{\bar \C}

\newcommand{\ti}{\tilde}
\newcommand{\nf}{\normalfont}

\newcommand{\tr}[1]{\textcolor{red}{#1}}

\allowdisplaybreaks

\title[Totally ramified value number greater than $2$]
{
Complete minimal surfaces of finite total curvature on punctured spheres
with totally ramified value number greater than $2$ 
}

\author{
Jun Matsumoto
}

\date{\today}



\address{
Department of Mathematics, \endgraf
Institute of Science Tokyo, \endgraf
O-okayama, Meguro, Tokyo, 152-8551, Japan
}
\email{matsumoto.j.d273@m.isct.ac.jp, j.matsumoto.517@gmail.com}

\begin{document}

\maketitle

\begin{abstract}
  Motivated by Osserman's problem on the number $D_g$ of omitted values 
  of the Gauss map of a complete minimal surface with finite total curvature in $\R^3$, 
  its totally ramified value number $\nu_g$ (referred to in this paper as the 
  \emph{total weight of totally ramified values})
  has attracted significant interest.
  The value of $\nu_g$ provides more detailed information than the number of omitted values alone.
  In 2006, Kawakami first found that a minimal surface defined on the three-punctured Riemann sphere, 
  originally constructed by Miyaoka and Sato, satisfies $D_g = 2$ and $\nu_g = 2.5 > 2$.
  Subsequently, in 2024, Kawakami and Watanabe gave another minimal surface 
  defined on the four-punctured Riemann sphere that also satisfies $D_g = 2$ and $\nu_g = 2.5$.
  To date, these remain the only two known examples of such surfaces satisfying $\nu_g > 2$.
  
  In this paper, we provide a systematic construction of meromorphic functions 
  on punctured Riemann spheres that satisfy $\nu_g > 2$.
  As a consequence, we obtain the following results for complete minimal surfaces
of finite total curvature with $\nu_g = 2.5$ 
within the topological types of the known examples:
(1) For the three-punctured sphere, we prove the uniqueness of Miyaoka--Sato's example.
(2) For the four-punctured sphere, we completely determine the surfaces with $D_g=2$ and $\nu_g = 2.5$,
which include examples other than Kawakami--Watanabe's one.
(3) Furthermore, we construct a new example on the four-punctured sphere
satisfying $D_g=1$ and $\nu_g = 2.5$.
\end{abstract}

\section{Introduction} 
  The study of the number of omitted values of the Gauss map of a complete minimal surface in Euclidean $3$-space 
  $\R^3$
  has been a central theme 
  in classical minimal surface theory.
  For this study, Fujimoto \cite{Fuj88} obtained a critical result
  by using value distribution theoretic methods:
  the Gauss map of a non-flat complete minimal surface
  in $\R^3$ can omit at most $4$ values from the Riemann sphere $\Chat \coloneqq \C \cup \{\infty\}$. 
  Moreover, this result is optimal since the Gauss map of the classical 
  Scherk surface omits exactly $4$ values in $\Chat$.
  Also, Osserman \cite{Oss64} proved that, 
  if a non-flat complete minimal surface in $\R^3$ is of finite total curvature, 
  then its Gauss map can omit at most $3$ values in $\Chat$ (see Fact \ref{fact:Osserman}).
  However, since it is still unknown whether there exist such minimal surfaces 
  whose Gauss map omits $3$ values,
  the optimality of Osserman's results remains open.
  
  In this context, the totally ramified value number
  of the Gauss map has attracted attention and been studied in the early 2000s. 
  Here, let us recall the notion of the totally ramified value number of meromorphic functions.
  
  \begin{definition}[cf. \cite{Nev70}, \cite{Kaw06}] \label{def:TRV}
  Let $g : \Sigma \to \Chat \coloneqq \C \cup \{\infty\}$ be a 
  non-constant meromorphic function on a Riemann surface $\Sigma$.
 \begin{itemize}[leftmargin = 5mm ]
	\item
	  A value in $\Chat \setminus g(\Sigma)$ is called an \emph{omitted value} of $g$.
	  We denote by $D_g$ the number of omitted values of $g$ in this paper.
	\item
	  Let $p \in \Sigma$.
	  Take local coordinates $z$ and $w$ centered at $p$ and $g(p)$ respectively, 
	  such that the function $g$ is locally expressed as $w = z^{e_p}$.
	  Then, we call the integer $e_p \in \Z_{\geq 1}$ the \emph{multiplicity} at $p$, and $e_p - 1$ the 
	  \emph{branching order} at $p$ of $g$, respectively.
	  Moreover, if $e_p \geq 2$, then the point $p$ is called  a \emph{branch point}.
	\item
	  A value $\al \in \Chat$ is called a \emph{totally ramified value} of $g$ 
	  if either it is an omitted value of $g$, or 
	  all the points in the inverse image $g^{-1}(\al)$ are branch points of $g$;
	  in other words, all solutions of the equation $g - \al = 0$ are multiple.
	\item
	  Let $\mathcal{TR}_g \coloneqq \{a_1, \dots, a_{D_g}, b_1, \dots, b_{R_g}\} \subset \Chat$ 
	  be the set of totally ramified values of $g$, where $a_k \ (k = 1, \dots, D_g)$ and  $b_j\ (j = 1, \dots, R_g)$
	  denote the omitted values and the totally ramified values other than the omitted values, respectively.
	  We also define a map $\nu : \mathcal{TR}_g \to \Z_{\geq 2} \cup \{\infty\}$ by
	  \begin{equation} \label{eq:order_TR}
	    \nu(\al)
	    \coloneqq
	    \left\{
	    \begin{aligned}
	      &\infty \qquad &(\al = a_k), \\
	      &\min \{e_p \in \Z_{\geq 2} ; p \in g^{-1}(b_j) \} \ (\geq 2) \qquad &(\al = b_j),
	    \end{aligned}
	    \right.
	  \end{equation}
	  and call the value $\nu(\al)$ the \emph{order} of a totally ramified value at $\al$.
	  We call $\nu_g$ defined by 
	  \begin{equation}
	    \nu_g \coloneqq \sum_{\al \in \mathcal{TR}_g}\left(1-\fr{1}{\nu(\al)}\right) 
	    = D_g + \sum_{j = 1}^{R_g}\left(1-\fr{1}{\nu(b_j)}\right) \in \boldsymbol Q_{\geq 0}
 	 \end{equation}
 	 the \emph{totally ramified value number}
 	 or \emph{total weight of totally ramified values} of $g$.
  \end{itemize}

  \end{definition}
  \begin{remark}
  Let us make a few remarks for notation and terms in this paper.
  \begin{itemize}[leftmargin = *]
    \item
        Hereafter, following Robinson \cite{Rob39}, and Kawakami and Watanabe \cite{KW24}, 
        we refer to the totally ramified value number as the ``total weight of totally ramified values''. 
        This is because the value $1 - 1 / \nu(\al)$ is called the weight 
        of a meromorphic function $g$  at $\al \in \mathcal{TR}_g$, and because
        we avoid confusing ``totally ramified value number'' with ``the number of totally ramified values''.
        We denote the number of totally ramified values other than omitted values by $R_g$.  
    \item
      The term ``exceptional value'' is also frequently used instead of the term ``omitted value''
      in the same sense.
      Regarding omitted values as totally ramified values is natural in Nevanlinna theory
      (see \cite{Rob39}, \cite{Nev70}, \cite{Fuj93}, 
      \cite{Kob03}, \cite{NW14}, \cite{Ru21}, \cite{Ru23}).
      Indeed, if $\alpha$ is an omitted value of $g$, 
      then the equation $g = \alpha$ has no solutions of any order, 
      which can be interpreted as having infinite multiplicity.
    \item
      The total weight of totally ramified values of a meromorphic function
      can be naturally interpreted from the viewpoint of Nevanlinna theory.
      In fact, the deficiency relation for the Nevanlinna defects (\cite{Nev25}) 
      implies
      $$
      	 D_g \leq \nu_g \leq 2
      $$
      for any non-constant meromorphic function $g : \C \to \Chat$.
   \end{itemize}
  \end{remark}
   
   Fujimoto implicitly showed that the total weight $\nu_g$ of totally ramified values 
   of the Gauss map $g$ 
   of a complete minimal surface provides more detailed information
   than the number $D_g$.
  Namely, \cite{Fuj92} 
  gives 
  \begin{equation} \label{eq:Fujimoto_case}
    D_g \leq \nu_g \leq 4
  \end{equation}
  for complete minimal surfaces (not necessarily of finite 
  total curvature), which is the best possible bound. 
  On the other hand, there are no known examples of complete minimal surfaces of 
  finite total curvature that satisfy $D_g = 3$, 
  while many examples with $D_ g = 2$ are obtained by Miyaoka and Sato \cite{MS94}
  in almost all topological types.
  Hence, as in the case of Fujimoto's work,  
  it was widely believed that the best upper bound for Osserman's problem is $2$, 
  and that of the total weight of totally ramified values of the Gauss map is also $2$.
 
 However, in 2006, Kawakami \cite{Kaw06} showed that a complete minimal surface 
  $\Chat \setminus \{\text{$3$ points}\} \to \R^3$ of total curvature $-8 \pi$ 
  given by Miyaoka and Sato satisfies $D_g = 2$ and $\nu_g = 2.5$, strictly greater than $2$.
  This finding gives a motivation for studying the total weight of totally ramified values of the Gauss map.
  By refining Osserman's arguments, Kawakami, Kobayashi and Miyaoka \cite{KKM08} showed that,
  assuming finite total curvature,  
  \begin{equation}
  	D_g \leq \nu_g < 4,
  \end{equation}
  which is different from Fujimoto's case \eqref{eq:Fujimoto_case}
  (see Fact \ref{KKM_TW_estimate} for a precise statement).
  Moreover, Kawakami and Watanabe \cite{KW24} found the second family 
  of complete minimal surfaces $\Chat \setminus \{\text{$4$ points}\} \to \R^3$ 
  of total curvature $-16 \pi$, which satisfies $D_g = 2$ and $\nu_g = 2.5$.
  The value $\nu_g = 2.5$ is best possible for the topological types of Miyaoka--Sato's  
  and Kawakami--Watanabe's examples.
  However, over the course of twenty years, 
  only these two examples with $\nu_g = 2.5$
   ---and thus strictly greater than $2$--- have been known, 
  and their uniqueness remains unknown.
  
  The purpose of this paper is to answer the following question.
  \begin{question} \label{Question}
    \begin{enumerate}[label = (\arabic*), leftmargin = *, font = \nf]
    	\item \label{Question_1}
    	  Are there other examples of  complete minimal immersions 
    	  $\Chat \setminus \{\text{$3$ points}\} \to \R^3$ of total curvature $-8\pi$
    	  that satisfy $\nu_g = 2.5$?
    	\item \label{Question_2}
    	    Are there other examples of complete minimal immersions 
    	    $\Chat \setminus \{\text{$4$ points}\} \to \R^3$  of total curvature $-16 \pi$
    	    that satisfy $\nu_g = 2.5$?
    \end{enumerate}
  \end{question}
  \noindent
   The answer to \ref{Question_1} is NO, namely, 
   we can prove the uniqueness of Miyaoka--Sato's case.
   However,  for \ref{Question_2}, the answer is YES, that is, 
   there are some examples  other than Kawakami--Watanabe's one.
   In particular, we can find a family containing their examples.
   This paper is also devoted to providing a method for constructing 
   complete minimal surfaces with finite total curvature satisfying $\nu_g > 2$.
   
   The paper is organized as follows. 
   In Section \ref{sec:preliminaries} we give a brief review of the fundamental theory of 
   classical minimal surface theory in $\R^3$, and describe Osserman's problem
   (Problem \ref{pro:Osserman}).
   Especially, we explain Miyaoka--Sato's and Kawakami--Watanabe's examples 
   (Examples \ref{ex:MS} and \ref{ex:KW})
   that served as the motivation for the Question \ref{Question}.
   In Section \ref{sec:estimate}, we review value distribution for meromorphic functions
   defined on punctured compact Riemann surfaces.
   By using an elementary estimate for the sum of orders \eqref{eq:order_TR} of totally ramified values
   (Proposition \ref{thm:estimate_order}),
   we obtain an estimate for $\nu_g$ 
   (Corollaries \ref{thm:nug_estimate} and \ref{cor:surj_25}). 
   In particular, Corollary \ref{cor:surj_25} immediately gives the 
   answer to \ref{Question_1} in Question \ref{Question}
   (Corollary \ref{thm:answer_MS}).
   In Section \ref{sec:KW_determining}, we completely determine 
   all complete minimal immersions
   $\Chat \setminus \{\text{$4$ points}\} \to \R^3$ of total curvature 
   $-16 \pi$ with $D_g = 2$ and $\nu_g = 2.5$,
   which corresponds to Kawakami--Watanabe's case.
   By giving a systematic construction of meromorphic functions on
   $\Chat \setminus \{\text{$4$ points}\}$ with $D_g = 2$ and 
   $\nu_g = 2.5$, we know that there are examples of such surfaces other 
   than Kawakami--Watanabe's  one (Theorem \ref{thm:TC16Pi_25}).
   We also find that there is a complete minimal immersion
   $\Chat \setminus \{\text{$4$ points}\} \to \R^3$ of total curvature $-16 \pi$ 
   with $D_g = 1$ and $\nu_g = 2.5$.
   Finally, we propose some open problems for the total weight of totally ramified values of 
   the Gauss map of a complete minimal surface in $\R^3$ 
   of finite total curvature (Problem \ref{pro:open_TRV}).
   
   \begin{acknowledgements}
     The author would like to sincerely express his gratitude to 
      Yu Kawakami and Kotaro Yamada
     for their helpful advice and comments on this research.
     The author also thanks Mototsugu Watanabe for sharing valuable ideas for this work.
     This work was supported by JST SPRING, Japan Grant Number JPMJSP2180.
   \end{acknowledgements}

\section{Preliminaries} \label{sec:preliminaries}

\subsection{
 Complete minimal immersions of finite total curvature into 
 \texorpdfstring{$\R^3$}{R3} and their Gauss maps
}
 
 Let us briefly review the fundamental concepts of minimal surface theory 
 (cf. \cite{Osserman}, \cite{LM99} and \cite{AFL21}).
 Let $\Sigma$ be a connected and orientable $2$-manifold, and $f: \Sigma \to \R^3$ an immersion.
 We denote by $ds^2$ the induced metric on $\Sigma$ by $f$ from $\R^3$.
 One can consider $\Sigma$ as a Riemann surface whose complex coordinate is isothermal regarding the induced 
 metric $ds^2$.
 Thus, in the sequel, we may suppose that $\Sigma$ is a Riemann surface and $f : \Sigma \to \R^3$ 
 is a conformal immersion.
 The conformal immersion $f$ is said to be a \emph{minimal immersion}
 if 
 $$
 \Delta_{ds^2} f \equiv 0,
 $$
 where $\Delta_{ds^2}$ denotes the Laplacian with respect to $ds^2$, 
 i.e., each component of $f$ is a harmonic function on $\Sigma$.
 As a classical result in minimal surface theory, the following Weierstrass representation is well known.
 
 \begin{fact}
   Let $\Sigma$ be a Riemann surface, and let a pair $(g, \omega)$ of a meromorphic function $g$ and
   a holomorphic $1$-form $\omega$ on $\Sigma$ satisfy
    \begin{enumerate}[label = (\arabic*), leftmargin = *, font = \normalfont]
      \item
      \emph{(regularity condition)}
      \begin{equation} \label{eq:1stff}
        ds^2 \coloneqq (1 + |g|^2)^2 |\omega|^2
       \end{equation}
       defines a Riemannian metric on $\Sigma$.
      \item
      \emph{(period condition)}
      \begin{equation}\label{eq:PC}
        \Re \int_c (1 - g^2, i(1 + g^2), 2g) \omega = 0
	\end{equation}
	holds for any closed curve $c$ in $\Sigma$.
    \end{enumerate}
    Then, a map $f : \Sigma \to \R^3$ defined by
    \begin{equation} \label{eq:W_rep}
      f \coloneqq \Re \int (1 - g^2, i(1 + g^2), 2g) \omega 
    \end{equation}
   gives a minimal immersion whose induced metric coincides with $ds^2$
   in \eqref{eq:1stff}.
   Conversely, any minimal immersion $f : \Sigma \to \R^3$ has an 
   expression \eqref{eq:W_rep} in which 
   the complex coordinate on $\Sigma$ is isothermal with respect to its induced metric.
 \end{fact}
 \noindent
 The pair $(g, \omega)$ of the meromorphic function $g$ and the holomorphic $1$-form $\omega$ 
 in the above fact is called the \emph{Weierstrass data} of a minimal immersion.
 
 The Gaussian curvature $K_{ds^2}$ of a minimal immersion $f : \Sigma \to \R^3$ is given by
 \begin{equation}
   K_{ds^2} = \fr{-4}{(1 + |g|^2)^4} \left|\fr{dg}{\omega}\right|^2.
  \end{equation}
 Additionally, the map $\mathcal G :\Sigma \to S^2$ defined by
 \begin{equation}\label{eq:C_Gauss_map}
   \mathcal G \coloneqq 
     \left(
       \fr{2\Re(g)}{1 + |g|^2}, \fr{2\Im(g)}{1 + |g|^2}, \fr{-1 + |g|^2}{1 + |g|^2} 
     \right)
     = \Pi^{-1} \circ g
 \end{equation}
 gives the (\emph{classical}) \emph{Gauss map} of $f$,
 where $\Pi : S^2 \to \Chat $ is the stereographic projection of the 
 $2$-dimensional unit sphere $S^2$ centered at the origin.
 Through the above identification between $\mathcal G$ and $g$, 
 the meromorphic function $g : \Sigma \to \Chat$ is also called the (\emph{complex}) \emph{Gauss map}.
 Moreover, the total curvature $C(\Sigma)$ of $f$ is given by
 \begin{equation}
   C(\Sigma) \coloneqq
   \int_\Sigma K_{ds^2} dA_{ds^2} 
   = - \int_\Sigma \fr{2i dg \wedge d \bar g}{(1 + |g|^2)^2} 
   = - \int_\Sigma g^\ast (dA_{S^2})
   \in \lbrack - \infty, 0 \rbrack,
  \end{equation} 
  where $dA_{ds^2}$ denotes the area element with respect to $ds^2$, 
  $$
    dA_{S^2} \coloneqq \fr{2i dz \wedge d \bar z}{(1 + |z|^2)^2}
  $$
  is the area element regarding the Fubini--Study metric on $S^2 = \Chat$, 
  and $z$ is the standard coordinate of $\C \subset \Chat$. 
  We say that a minimal immersion is of \emph{finite total curvature} if 
  the total curvature satisfies $C(\Sigma) > - \infty$.
  Moreover, we say that a minimal immersion is \emph{complete}
  if the induced metric is a complete Riemannian metric on $\Sigma$.

  Let us next review fundamental facts about complete minimal surfaces of finite total 
  curvature.
  
  \begin{fact}[{\cite{Hub57}, \cite{Oss64}}]
    Let $\Sigma$ be a Riemann surface, and $f : \Sigma \to \R^3$ a minimal immersion.
    Then, the following two conditions are equivalent:
    \begin{enumerate}[label = (\arabic*), leftmargin = *, font = \normalfont]
      \item
        The minimal immersion $f : \Sigma \to \R^3$ is complete and of finite total curvature.
      \item
        The source Riemann surface $\Sigma$ is biholomorphic to 
        $\BSigma \setminus \{p_1, \dots, p_n\}$, where $\BSigma$ is a compact Riemann surface and 
        $n \geq 1$ is an integer. 
        Moreover, the $\C^3$-valued holomorphic $1$-form $(1 - g^2, i(1 + g^2), 2g) \omega$ 
        {\nf(}also $(g, \omega)${\nf)} on $\Sigma$ 
        extends meromorphically to $\BSigma$. 
        
    \end{enumerate}
    
        In particular, if $f : \Sigma \to \R^3$ is a complete minimal immersion of finite total curvature, 
        then the Gauss map $g : \Sigma \to \Chat$ of $f$ extends to a holomorphic map $g : \BSigma \to \Chat$, 
        and the total curvature is given by $C(\Sigma) = -4\pi \deg(g)$, 
        where $\deg(g)$ denotes the degree of $g : \BSigma \to \Chat$. 
  \end{fact}
  
  \subsection{
    Osserman's problem and its relationship with the total weight of totally ramified values of the Gauss map 
  }
  
  Let $f : \Sigma = \BSigma_\gm \setminus \{p_1, \dots, p_n\} \to \R^3$
  be a complete minimal immersion of finite total curvature,
  where $\BSigma_\gm$ stands for a compact Riemann surface of genus $\gm$.
   Osserman proved the following results 
  and proposed a problem about the number $D_g$ of omitted values
  of the Gauss map $g : \Sigma \to \Chat$.
  
  \begin{fact}[{\cite[Theorem 3.3]{Oss64}}] \label{fact:Osserman}
  The Gauss map of a non-flat complete minimal immersion of finite total curvature in 
  $\R^3$ satisfies $D_g \leq 3$.
  \end{fact}
  
  However, the optimality of Osserman's result is still unknown.
  So, he proposed the following problem in his book, which remains open.
  
  \begin{problem}[{\cite[\S 9]{Osserman}}] \label{pro:Osserman}
    Does there exist a non-flat complete minimal immersion
    of finite total curvature in $\R^3$, which satisfies $D_g = 3$?
  \end{problem}
  
  \noindent
    If the answer to Osserman's problem is NO , then the best upper bound is $2$
    since the catenoid satisfies $D_g = 2$, for instance.
    Additionally, Miyaoka and Sato \cite{MS94} showed that there are 
    complete minimal surfaces of finite total curvature with $D_g = 2$ 
    in almost all topological types.
    If there exist such immersions with $D_g = 3$, then (1) 
    $|C(\Sigma)| \geq 16 \pi$ (\cite{WX87}, \cite{Fan93}), 
   (2) $\gm \geq 1$ and $\deg(g) \geq n \geq 3$ (\cite{Oss64}) hold.
  
  On the other hand, 
  if any complete minimal immersion of finite total curvature satisfies $\nu_g < 3$,
  then the answer to Osserman's problem is negative, where 
  $\nu_g$ stands for the total weight of totally ramified values of the Gauss map. 
  As mentioned in the Introduction, in light of Fujimoto's result (\cite{Fuj89}), it
  was expected that $D_g \leq \nu_g \leq 2$.
  However, Kawakami found that the following example satisfies $\nu_g  = 2.5 > 2$.
  
  \begin{example}[{\cite[Proposition 3.1]{MS94}, \cite[Theorem 2.1]{Kaw06}}] \label{ex:MS}
  	Let $\Sigma = \Chat \setminus \{\pm i, \infty\}$, and define a pair $(g, \omega)$ 
  	of a meromorphic function $g$ and a holomorphic $1$-form $\omega$ on $\Sigma$ by
  	\begin{equation}
  		\left\{
  		\begin{array}{l}
	  		g = \sigma \dfrac{z^2 + 1 + a(t - 1)}{z^2 + t}, \\
	  		\\
	  		\omega  = \dfrac{(z^2 + t)^2}{(z + i)^2(z - i)^2} dz,
  		\end{array}
  		\right.
  	\end{equation}
  	where $a, t \in \R \setminus \{1\}$ and $\sigma \in \C \setminus \{0\}$ satisfy $(a - 1)(t -1) \neq 0$, $a \neq 0$,
  	and 
  	$$
  	\sigma^2 = \fr{t + 3}{a((t-1) a + 4)} < 0.
  	$$
  	Then, the pair $(g, \omega)$ gives a family of complete minimal immersions 
  	with total curvature $-8 \pi$.
  	
  	The Gauss map $g : \Sigma \to \Chat$ of a minimal immersion in this family has two omitted values 
  	$g(\infty) = \sigma$ and $g(\pm i) = \sigma a$, 
  	and the unique totally ramified value $g(0) = \sigma(1 + a(t - 1))/t$ of order $2$.
  	Thus, each minimal immersion in this family has the total weight of totally ramified values 
  	$\nu_g = 2 + (1 - 1/2) =  2.5$.
  	
  	\cite[Theorem 2]{MS94} states that all complete minimal immersions
  	$\Chat \setminus \{\text{$3$ points}\} \to \R^3$ with $C(\Sigma) = - 8 \pi$ and $D_g = 2$ 
  	are given by this family. 
  
  \end{example}
  
  The following fact is an estimate
  between the number of omitted values 
  and the total weight of totally ramified values of the Gauss map.
  
  \begin{fact}[{\cite[Theorem 3.3]{KKM08}, \cite[Theorem 3.9]{KW24}}] \label{KKM_TW_estimate}
     Let $f: \Sigma \coloneqq \BSigma_\gm \setminus \{p_1, \dots, p_n\} \to \R^3$ 
     be a non-flat complete minimal 
     immersion of finite total curvature, $g : \Sigma \to \Chat$ the Gauss map. 
     Then, the number $D_g$ of omitted values and the total weight $\nu_g$ of
     totally ramified values of $g$ satisfy
     \begin{equation} \label{eq:TRBN_est}
       D_g \leq \nu_g \leq 2 + \fr{2}{R} < 4
       \qquad
       \left(
       		\fr{1}{R} \coloneqq \fr{\gm - 1 + (n / 2)}{\deg (g)} = \fr{-\chi(\BSigma_\gm) + n}{2 \deg(g)}
	\right),
     \end{equation}
     where $\chi(\BSigma_\gm) = 2 - 2\gm$ is the 
     Euler characteristic of $\BSigma_\gm$.
     Furthermore, if $\gm = 0$, i.e., $\BSigma_0 = \Chat$, then it holds that
     \begin{equation}
     	D_g \leq \nu_g <3,
     \end{equation}
     and thus 
     $
     	D_g \leq 2.
     $
  \end{fact}

  Moreover, recently, Kawakami and Watanabe constructed 
  the following example that also satisfies $\nu_g = 2.5 > 2$.
  
  \begin{example}[{\cite[Theorem 3.7]{KW24}}] \label{ex:KW}
    Let $\Sigma = \Chat \setminus \{0, \pm i, \infty\}$, and define a pair $(g, \omega)$
    of a meromorphic function $g$ and a holomorphic $1$-form $\omega$ on $\Sigma$ by 
	\begin{equation}
			\left\{
			\begin{array}{l}
				g = \sigma \dfrac{(b-a) z^4 + 4 a (b-1) z^2 + 4 a (b-1)}{(b-a) z^4 + 4 (b-1) z^2 + 4 (b-1)},\\
				\\
				\omega = \dfrac{((b-a) z^4 + 4(b-1) z^2 + 4 (b-1))^2}{z^2 (z + i)^2(z-i)^2} dz,
			\end{array}
			\right.
	\end{equation} 
   where $a, b \in \R \setminus \{1\}$ with $a \neq b$, 
   and $\sigma \in i \R \setminus\{0\}$ satisfy
	\begin{equation} \label{eq:KW_period}
 		\sigma^2 = \fr{5a + 11b - 16}{16ab - 11a - 5b} < 0.
	\end{equation}
   Then, the pair $(g, \omega)$ gives a family of complete minimal immersions with $C(\Sigma) = -16 \pi$.
 
   The Gauss map $g : \Sigma \to \Chat$ of a minimal immersion in this family has two omitted values 
   $g(\infty) = g(\pm i) = \sigma$ and $g(0) = \sigma a$, 
   and the unique totally ramified value $g(\pm \rt{2} i) = \sigma b$ of order $2$.
   Thus, each minimal immersion in this family has the total weight of totally ramified values 
  $\nu_g = 2 + (1 - 1/2) =  2.5$.
  \end{example}

  Up to the present day, only the above two examples of 
  complete minimal immersions of finite total curvature with $\nu_g = 2.5 > 2$ have been known.
  Moreover, the families in Examples \ref{ex:MS} and \ref{ex:KW} give a sharp estimate 
  in \eqref{eq:TRBN_est} when $\gm = 0, n = 3$, and $\deg(g) = 2$,
  and when $\gm = 0, n = 4$, and $\deg(g) = 4$.
  Therefore, in this paper, 
  we discuss the existence of other examples with the same topological type as these two examples
  (see Question \ref{Question}).

\section{
  Some Estimates for orders and total weight of totally ramified values 
  of meromorphic functions on punctured Riemann surfaces
} \label{sec:estimate}

 In this section, we discuss estimates for the orders of totally ramified values of meromorphic functions
 on punctured compact Riemann surfaces in prder
 to apply them to the construction of meromorphic functions.
 
 \begin{proposition} \label{thm:estimate_order}
   Let $\BSigma$ be a compact Riemann surface, and $p_1, \dots, p_n \in \BSigma\ (n \in \Z_{\geq 0})$.
   Let $g : \BSigma \setminus \{p_1, \dots, p_n\} \to \Chat$ be a meromorphic function.
   Assume that the function $g$ extends meromorphically to $\BSigma$, i.e., $g : \BSigma \to \Chat$ is a 
   holomorphic map between compact Riemann surfaces, and suppose $d \coloneqq \deg(g) \geq 2$.
   Let $\mathcal{TR}_g \coloneqq \{a_1, \dots, a_{D_g}, b_1, \dots, b_{R_g}\} \subset \Chat$ 
   be the set of totally ramified values of $g$, where $a_k \ (k = 1, \dots, D_g)$ and  $b_j\ (j = 1, \dots, R_g)$
   denote the omitted values and 
   the totally ramified values other than the omitted values, respectively
    (see Definition \ref{def:TRV}).
   Then, the sum of orders of the totally ramified values other than the omitted values satisfies
   \begin{equation} \label{eq_prop:sum estimate}
   	2 R_g \leq S_g \leq \left \lfloor \fr{d}{d - 1} B_g \right \rfloor,
   \end{equation}
   where
   \begin{equation} \label{eq:Tdef}
   	 S_g \coloneqq \sum_{j = 1}^{R_g} \nu(b_j),
   	 \qquad
   	 B _g \coloneqq \sum_{p \in \BSigma} (e_p - 1) =  - \chi(\BSigma) + 2d,
   \end{equation}
   $e_p$ is the multiplicity of $g$ at $p$, $\nu(b_j)$ is the order of a totally ramified value at $b_j$
   {\nf (}see Definition \ref{def:TRV}{\nf )},
   and $\chi(\BSigma)$ is the Euler characteristic of $\BSigma$.
   
 \end{proposition}
 
 Under the notation and assumptions of Proposition \ref{thm:estimate_order}, 
 let us prove the following two lemmas.
 
 \begin{lemma} \label{lem:estimate_order_1}
   Let $b \in \mathcal{TR}_g \setminus \{a_1, \dots, a_{D_g}\}$, and set $\nu_b \coloneqq \nu(b)$.
   Then, it holds that
   \begin{equation} \label{eq:c(nu)}
   	\sum_{p \in g^{-1}(b)} (e_p - 1) \geq c(\nu_b) \qquad 
   	\left(
   		c(\nu) \coloneqq d - \left\lfloor \fr{d}{\nu} \right \rfloor \ (\nu \in \Z_{\geq 2})
   	\right).
   \end{equation}
 \end{lemma}
 
 \begin{proof}
   By definition of the order $\nu : \mathcal{TR}_g \to \Z_{\geq 2} \cup \{\infty\}$, 
   we have $e_p \geq \nu_b$.
   Let $m \in \Z_{\geq 1}$ be the cardinality of $g^{-1}(b)$.
   Since
   $$
   	d = \sum_{p \in g^{-1}(b)} e_p \geq \sum_{p \in g^{-1}(b)} \nu_b = m \nu_b,
   $$
   and $m$ is an integer, we have $m \leq \lfloor d / \nu_b \rfloor$.
   Thus, we find that 
   $$
   	\sum_{p \in g^{-1}(b)} (e_p - 1) = d - m \geq d - \left \lfloor \fr{d}{\nu_b} \right \rfloor = c(\nu_b). 
   $$
 \end{proof}
 
 \begin{lemma} \label{lem:estimate_order_2}
   Let $b \in \mathcal{TR}_g \setminus \{a_1, \dots, a_{D_g}\}$, and set $\nu_b = \nu(b)$.
   Then, it holds that
   \begin{equation}
   	\nu_b \leq \fr{d}{d - 1} c(\nu_b).
   \end{equation}
 
 \end{lemma}
 
 \begin{proof}
   By $\lfloor d / \nu_b \rfloor \leq d / \nu_b$ and definition of $c(\nu_b)$, one can observe
   $$
   	c(\nu_b) \geq d - \fr{d}{\nu_b} = \fr{d (\nu_b - 1)}{\nu_b},
   $$ 
   and so
   \begin{equation} \label{lem_eq_c(nu_b)}
   	\fr{d}{d - 1} c(\nu_b) \geq \fr{d^2 (\nu_b - 1)}{(d - 1) \nu_b }.
   \end{equation}
   Since $2 \leq \nu_b \leq d$, we have 
   \begin{align*}
   	d^2 (\nu_b - 1) - (d - 1) \nu_b^2 = - (\nu_b - d) ((d - 1) \nu_b - d) 
   	&\geq - (\nu_b - d)((d - 1) \cdot 2 - d) \\
   	&= (d - \nu_b)(d - 2)\geq 0.
   \end{align*}
   Hence, we know that
   $$
   	\fr{d^2 (\nu_b - 1)}{(d - 1) \nu_b } \geq \nu_b,
   $$
   which implies the conclusion by \eqref{lem_eq_c(nu_b)}.
 \end{proof}
 
 \begin{proof}[Proof of Proposition \ref{thm:estimate_order}]
  The first inequality in \eqref{eq_prop:sum estimate} follows from
  $\nu(b_j) \geq 2$.
  
   Let us show the second inequality in \eqref{eq_prop:sum estimate}.
   By Lemma \ref{lem:estimate_order_1}, 
   $$
   	\sum_{j = 1}^{R_g} c(\nu(b_j)) \leq \sum_{j = 1}^{R_g} \sum_{p \in g^{-1}(b_j)} (e_p - 1)
   	 = \sum_{p \in \bigsqcup_{j = 1}^{R_g} g^{-1}(b_j)} (e_p - 1)
   	\leq \sum_{p \in \BSigma} (e_p - 1) =  B _g.
   $$
   This inequality and Lemma \ref{lem:estimate_order_2} yield
   $$
   	S_g = \sum_{j = 1}^{R_g} \nu(b_j) \leq \fr{d}{d - 1} \sum_{j = 1}^{R_g} c(\nu(b_j)) \leq \fr{d}{d - 1}  B _g,
   $$
   which implies the conclusion.
 \end{proof}
 
 As a corollary of Proposition \ref{thm:estimate_order}, one can easily deduce the following assertion.
 
 \begin{corollary} \label{cor:order_estimate}
   If $\BSigma = \Chat$ in Proposition \ref{thm:estimate_order}, then it holds that
   \begin{equation} \label{eq_cor:total order}
   	2 R_g \leq S_g \leq 2d.
   \end{equation}
 \end{corollary}
 
 The inequality \eqref{eq_prop:sum estimate} and especially \eqref{eq_cor:total order}
 are sharp in the following sense. 
 Let a meromorphic function $g: \Chat \to \Chat $ be $g(z) = z^d \ (d \geq 2)$, and it satisfies them.
 In fact, this $g$ has two totally ramified values $0$ and $\infty$ and is of degree $d$.
 Hence, we see that $\nu(0) + \nu(\infty) = d + d = 2d$.
 Moreover, for the total weight of totally ramified values of meromorphic functions, 
 we find the following estimate.
 
 \begin{corollary} \label{thm:nug_estimate}
   Under the same assumption and notation as Proposition \ref{thm:estimate_order}, 
   the total weight of totally ramified values of a meromorphic function $g$ satisfies
   \begin{equation} \label{eq_thm:nug_estimate}
   	\nu_g  \leq D_g + \fr{1}{4} \left \lfloor \fr{d}{d - 1} B_g \right \rfloor. 
   \end{equation} 
 \end{corollary}
 
 \begin{proof}
   By definition of $\nu_g$ and Proposition \ref{thm:estimate_order}, we have
   \begin{align*}
 	\nu_g 
 	= D_g + R_g - \sum_{j = 1}^{R_g} \fr{1}{\nu(b_j)} 
 	\leq D_g + R_g - \fr{R_g^2}{S_g}
 	&\leq D_g + \fr{S_g}{4}\\
 	& \leq D_g + \fr{1}{4} \left \lfloor \fr{d}{d - 1} B_g \right \rfloor.
    \end{align*}
    In fact, the first inequality follows from Cauchy--Schwarz's inequality.
    The second inequality is obtained by $R_g - R_g^2 / S_g \leq S_g/4$ since
     $0 \leq R_g \leq S_g / 2$.
 \end{proof}
 
 \begin{corollary} \label{cor:surj_25}
   In Corollary \ref{thm:nug_estimate}, if $\BSigma = \Chat$ and $d \leq 4$, then there is
   no surjective meromorphic functiong (i.e., $D_g = 0$) satisfying $\nu_g  > 2$.
   Moreover, if $\BSigma = \Chat$ and $d = 2$, then there is no
   meromorphic function with $D_g \leq  1$ satisfying $\nu_g  > 2$.
 \end{corollary}
 
 \begin{proof}
   Assume that there exists such a meromorphic function $g$.
   Then, $\BSigma = \Chat, d \leq 4, D_g = 0$, and \eqref{eq_thm:nug_estimate} imply
   $$
   	\nu_g \leq \fr{1}{4} \left \lfloor \fr{d}{d - 1}\cdot (- 2 + 2d) \right \rfloor = \fr{1}{2} d \leq 2,
   $$
   which is a contradiction.
   
   Next, we show the latter part. 
   If there exists such a meromorphic function, then, by $d = 2$ and $D_g \leq 1$, we see that
   $$
   	\nu_g \leq 1 + \fr{1}{4} \left \lfloor \fr{2}{2 - 1}\cdot (-2 + 2 \cdot 2 ) \right \rfloor = 2,
   $$
   which is also a contradiction, and we obtain the conclusion.
 \end{proof}
 
  Using the latter part of this Corollary and the fact stated in Example \ref{ex:MS}, 
 we obtain the following corollary, which gives the answer to \ref{Question_1} in Question \ref{Question}.
 
 \begin{corollary} \label{thm:answer_MS}
   A complete minimal immersion $\Chat \setminus \{\text{$3$ points}\} \to \R^3$ with 
   total curvature $-8\pi$ whose total weight of totally ramified values of 
   the Gauss map satisfies
   $\nu_g = 2.5$ is given by the family in Example \ref{ex:MS}. 
 \end{corollary}
 
 At the end of this section, we state an estimate for the number $R_g$ of totally ramified values 
 other than omitted values of a meromorphic 
 function defined on a punctured Riemann surface.
 
 \begin{corollary} \label{cor:R_g_estimate}
   Under the same assumptions and notation as Proposition \ref{thm:estimate_order}, 
   the number $R_g$ of totally ramified values other than omitted values of a meromorphic function $g$ satisfies
   \begin{equation}
     R_g \leq    
         \left\lfloor \fr{B_g}{\lceil d/2 \rceil} \right \rfloor.
   \end{equation}
 \end{corollary}
 
 \begin{proof}
   Let $b_j \in \mathcal{TR}_g \setminus \{a_1, \dots, a_{D_g}\}\ (j = 1, \dots R_g)$.
   Since $c(\nu)$ in \eqref{eq:c(nu)} is monotonically increasing and it takes the minimum value 
   $c(2) = d - \lfloor d/2 \rfloor = \lceil d/2 \rceil$, by Lemma \ref{lem:estimate_order_1}, we have
   $$
     B_g \geq \sum_{j = 1}^{R_g} \sum_{p\in g^{-1}(b_j)}(e_p - 1) \geq \sum_{j = 1}^{R_g} c(\nu(b_j)) 
     \geq \sum_{j = 1}^{R_g} c(2) = R_g \left \lceil \fr{d}{2} \right \rceil,
   $$
   which implies the conclusion.
 \end{proof}
 
\section{
 Complete minimal surfaces on four-punctured Riemann spheres  
 with $C(\Sigma) = -16\pi$ and $\nu_g = 2.5$
} \label{sec:KW_determining}

 In this section, we consider complete minimal immersions $\Chat \setminus \{\text{$4$ points}\} \to \R^3$
 with total curvature $-16\pi$ and total weight $\nu_g = 2.5$ of 
 totally ramified values of the Gauss map.
 
 Let us consider a complete minimal immersion 
 $f : \Sigma \coloneqq \Chat \setminus \{\text{$4$ points}\} \to \R^3$ that satisfies such conditions.
 By a coordinate change of the source Riemann surface $\Sigma$, 
 we may assume that 
 $\Sigma = \Chat \setminus \{\infty, \pm i, t\} \ (t \in \Chat \setminus \{\infty, \pm i\}).$
 By a suitable rotation in $\R^3$, we may also suppose that $g$ has no pole at each end.
 In addition, since $\deg(g) = 4$, Riemann--Hurwitz's formula gives
 \begin{equation}
  B_g = 6.
 \end{equation}
 Let $\mathcal{TR}_g = \{a_1, \dots, a_{D_g}, b_1, \dots, b_{R_g}\} \subset \Chat$ 
 be the set of totally ramified values of the Gauss map $g$, where $a_k \ (k = 1, \dots, D_g)$ and  
 $b_j\ (j = 1, \dots, R_g)$
 denote the omitted values and the totally ramified values other than the omitted values, respectively.
 By Corollary \ref{cor:R_g_estimate}, the number $R_g$ satisfies 
 \begin{equation} \label{eq:R_g3}
   R_g \leq 3.
 \end{equation}
  In addition, by Corollary \ref{cor:order_estimate}, note that
 \begin{equation} \label{eq:nu2ijo}
 	2 R_g \leq \sum_{j = 1}^{R_g} \nu_j \leq 8,
 \end{equation}
 where $\nu_j \coloneqq \nu(b_j)\  (j = 1, \dots, R_g)$.
 
  \subsection{
 The case where $D_g = 2$
 }
 
 In this case, the total weight $\nu_g$ of totally ramified values of the Gauss map is
 $$
 	\nu_g = 2 + \sum_{j = 1}^{R_g} \left (1 -\fr{1}{\nu(b_j)}\right).
 $$ 
 If $R_g \geq 2$, then $\nu_g$ satisfies
 $$
 	\nu_g = (2 + R_g) - \sum_{j = 1}^{R_g} \fr{1}{\nu(b_j)} \geq (2 + R_g) - \fr{1}{2} R_g = 2 + \fr{1}{2} R_g \geq 3 > 2.5,
 $$ 
 which is a contradiction.
 Hence, since we see $R_g = 1$,
 set $\mathcal{TR}_g = \{\sigma, \tau, b\}$, 
 where $\sigma, \tau \in \Chat \setminus \{\infty\} = \C$ are omitted values, and $b \in \Chat $ 
 is the unique totally ramified value, which are all distinct.
 Also since $\deg(g) = 4$ and $\nu_g = 2.5$, we can set $g^{-1}(b) = \{q_1, q_2\}$, 
 where $q_1, q_2 \in \Chat$ 
 are distinct values
 that satisfy 
 \begin{equation}\label{eq:order_RV_5.2}
   \ord_{q_l}(g - b) = 2 \qquad  (l = 1, 2),
  \end{equation}
  and hence the order of 
 the totally ramified value $b$ is $\nu(b) = 2$.
 
 Here, let a M\"{o}bius transformation $\Phi : \Chat \to \Chat$ be defined by
 \begin{equation} \label{eq:Mobius_2}
   \Phi(w) \coloneqq \fr{\tau - b}{\tau - \sigma} \cdot \fr{w - \sigma}{w - b},
 \end{equation}
 which satisfies
 $$
 \Phi(\sigma) = 0, \qquad \Phi(\tau) = 1, \qquad \Phi(b) = \infty.
 $$
 Now, let 
 \begin{equation}
 G\coloneqq \Phi \circ g : \Chat \setminus \{\infty, \pm i, t\} \to \Chat.
 \end{equation}
 Then, the meromorphic function $G$ omits the values $\{0, 1\}$ and has a totally ramified value $\infty$.
 Since the Gauss map $g$ has no pole at any end, 
 the function $G$ has no poles at the ends.
 In addition, since $\deg(g) = \deg(G) = 4$ and $g^{-1}(b) = \{q_1, q_2\}$ satisfying \eqref{eq:order_RV_5.2},
 we have
 \begin{equation} \label{eq:asum_G}
   G^{-1}(\infty) = \{q_1, q_2\}, \qquad
   \ord_{q_l} G = -2,\qquad 
   \{q_1, q_2\} \cap \{\infty, \pm i, t\} = \varnothing.
 \end{equation}
 Hence, all zeros of $G$ and $G-1$ occur at the ends.
 
 Next, let us classify the allocation of the ends to the omitted values $0$ and $1$ of $G : \Sigma \to \Chat$,
 together with their multiplicities. 
 For brevity, we write
 $$
 	0 : [\lambda_1, \lambda_2, \lambda_3, \lambda_4]
 $$
 to mean that the multiplicity of $0$ of $G$ at ends are 
 $\lambda_1, \lambda_2, \lambda_3,$ and  $\lambda_4$ (no particular order).
 Similarly, 
 $$
 	1 : [\mu_1, \mu_2, \mu_3, \mu_4]
 $$ 
 for that of $1$ of $G$;
 in other words, each $\mu_m\ (m = 1, 2, 3, 4)$ is the order of zero of $G-1$ at an end.
 We always assume $\lambda_m, \mu_m \in \Z_{\geq 0}$, and then
 $$
 	\sum_{m = 1}^4 \lambda_m = \sum_{m = 1}^4 \mu_m = \deg(G) = 4
 $$
 holds.
 An entry $\lambda_m = 0$ (resp. $\mu_m = 0$) indicates that the corresponding end is not 
 contained in the fiber of $0$ (resp. $1$) by $G$;
 for convenience, zero entries are omitted from the bracket notation.
 Additionally, The notation $G(p) = q^{(k)}$ means that $\ord_p(G - q) = k \in \Z_{\geq 1}$, 
    and $G(p) = \infty^{(k)}$ means that $\ord_p G = - k \in \Z_{< 0}$. 
 
 Under these notations, we must consider the following four cases:
 \begin{center}
   {\bf Case 1}: $0 : [4], \quad 1 : [2, 1, 1]$, \qquad {\bf Case 2}: $0 : [3, 1], \quad 1 : [3, 1]$,\\
   {\bf Case 3}: $0 : [3, 1], \quad 1 : [2, 2]$,\qquad {\bf Case 4}: $0 : [2, 2], \quad 1 : [2, 2]$.
 \end{center}
 We will show that the above cases reduce to only three cases by taking the totally ramified values into consideration.
  \noindent
  Let us first show that if there exists a meromorphic function to satisfy one of {\bf Case 1} through {\bf Case 4} above, 
  then the structure of the source Riemann surface can be essentially determined as follows.
  
  \begin{lemma} \label{lem:structure_of_source}
    Let $E \coloneqq \{\infty, \pm i, t\}$, and let $G : \Sigma = \Chat \setminus E \to \Chat$ be a meromorphic function.
    Assume that the meromorphic function $G$ extends to $\Chat$ meromorphically, 
    $\deg G = 4$,  
    $G(\Chat \setminus E) = \Chat \setminus \{0, 1\}$, 
    $\infty$ is the unique totally ramified value of $G$ 
    other than omitted values $0$ and $1$, and \eqref{eq:asum_G}.
    
    \begin{enumerate}[label = (\arabic*)]
     \item \label{sphere_structure_1}
       If $G$ satisfies the condition of {\bf Case 1}, then the source Riemann surface $\Sigma$ is 
       biholomorphic to $\Chat \setminus \{\infty, \pm i, 0\}$.
     \item \label{sphere_structure_2}
       If $G$ satisfies the condition of {\bf Case 2}, then $\Sigma$ is 
       biholomorphic to $\Chat \setminus \{\infty, \pm i, - 17 i\}$.
     \item
       There does not exist a meromorphic function $G$ satisfying the condition of {\bf Case 3}.
     \item \label{sphere_structure_4}
       If $G$ satisfies the condition of {\bf Case 4}, then the source Riemann surface $\Sigma$ is 
       biholomorphic to $\Chat \setminus \{\infty, \pm i, 0\}$.
   \end{enumerate}
  \end{lemma}
  
  \begin{proof}
    Note that $G$ satisfies $B_G = 6$.
    Assume that such a meromorphic function $G$ satisfies the condition in each case.
    We denote the cross ratio of $z_1, z_2, z_3, z_4 \in \Chat$ by 
    $$
      [z_1, z_2 ; z_3, z_4] \coloneqq \fr{z_1 -  z_3}{z_1 - z_4} \cdot \fr{z_2 - z_4}{z_2 - z_3}. 
    $$
    \begin{enumerate}[leftmargin = *, label = (\arabic*)]
      \item
        Note that $G^{-1}(0) \subset E$ consists of one point.
        Let $\phi_1 : \Chat \to \Chat$ be a M\"obius transformation to satisfy 
        $\phi_1(q_1) = 0, \phi_1(q_2) = \infty, \phi_1(G^{-1}(0)) = \{1\}$, 
        and set $\tilde G \coloneqq G \circ \phi_1^{-1}$.
        Then, the condition in {\bf Case 1} gives 
        $$
          \ti G = C \fr{(z - 1)^4}{z^2} \qquad (C \neq 0).
        $$
        Since the equation $\ti G - 1 = 0$ has a double solution at some point $z_0\neq 0, 1, \infty$ 
        and two simple solutions, 
        by $(\ti G - 1)' (z_0) =  2C (z_0 - 1)^3(z_0 + 1) / z_0^3$, 
        we see $z_0 = -1$ because of $z_0 \neq 1$. 
        Since $\ti G(-1) = 1$, we have $C = 1/16$, and thus 
        $$
          \ti G = \left( \fr{(z - 1)^2}{4 z} \right)^2.
        $$
        
        Next, let us determine $\phi_1(E)$. 
        By direct computation, we know $\ti G^{-1}(1) = \{ -1, \lambda, 1 / \lambda\}$, 
        where $\lambda = 3 + 2 \rt 2$, and hence $\phi_1(E) = \{\pm 1, \lambda, 1 / \lambda\}$.
        Since $[\lambda, 1 / \lambda ; 1, -1] = -1 = [\infty, 0; i, -i]$, there exists a M\"obius transformation 
        $\phi_2 : \Chat \to \Chat$ such that $\phi_2(\phi_1(E)) = \{\infty, \pm i, 0\}$.
        Hence, the map $\phi_2 \circ \phi_1$ restricting to $\Chat \setminus E$ 
        gives a biholomorphic map between $\Chat \setminus E$ and 
        $\Chat \setminus \{\infty, \pm i, 0\}$.
        
      \item
        We set $G^{-1}(0) = \{r_1, r_2\} \subset E$ and may suppose $G(r_1) = 0^{(3)}$.
        Let us take a M\"obius transformation $\phi_1 : \Chat \to \Chat$ 
        to satisfy $\phi_1(q_1) = 0, \phi_1(q_2) = \infty, \phi_1(r_1) = 1$
        and set $a \coloneqq \phi_1(r_2)\  (\neq 0, 1, \infty)$ and $\ti G = G \circ \phi_1^{-1}$.
        Then, we can set $\ti G$ as 
        $$
          \ti G = C \fr{(z - 1)^3 (z - a)}{z^2} \qquad (C \neq 0).
        $$
        
        Since the polynomial $C (z - 1)^3 (z - a) - z^2$, which is the numerator of $\ti G - 1$ has a triple root at some 
        $z_0 \neq 0, 1, a$, 
        the equation $(\ti G - 1)' = \ti G' = 0$ has a double solution at $z_0$.
        \begin{equation} \label{eq:tiG_(2)}
          \ti G' = \fr{2 z ^2 + (1 - a) z - 2 a}{z ( z - 1)(z - a)} \ti G
        \end{equation}
       and $G(z_0) \neq 0$ imply that since the polynomial $2 z ^2 + (1 - a) z - 2 a$ has a double root at $z_0$, 
       it holds that 
       $$
       \left\{
        \begin{array}{l}
          1 - a = -4 z_0, \\
          - a  = z_0^2,
        \end{array}
        \right.
       $$
       and thus one observes that $a = - z_0^2 $ and $z_0$ satisfies $z_0^2 + 4 z_0 + 1 = 0$. 
       On the other hand, let $z_1 \ (\neq z_0, 0, 1, a)$ be a simple root of $C (z - 1)^3 (z - a) - z^2$.
       Then, Vieta's formula gives $z_0 \cdot z_0 \cdot z_0 \cdot z_1 = (C a)/C = a$, and we know 
       $z_1 = - 1 / z_0$. 
       Hence, we have $\phi_1(E) = \{1, z_0, - z_0^2, - 1 / z_0\}. $
       
       By using $z_0^2 + 4 z_0 + 1 = 0$,
       one can see that
       $[1, - z_0^2 ; z_0, -1 / z_0] = 9 = [\infty, i ; - i, - 17 i]$. 
       Therefore, there exists a M\"obius transformation $\phi_2 : \Chat \to \Chat$ such that 
       $\phi_2(\phi_1(E)) = \{ \infty, \pm i, - 17 i\}$, and hence 
       the map $\phi_2 \circ \phi_1$ restricting to $\Chat \setminus E$ 
       gives a biholomorphic map between $\Chat \setminus E$ and 
        $\Chat \setminus \{\infty, \pm i, - 17 i\}$.
      \item
        We denote $G^{-1}(0) = \{r_1, r_2\} \subset E$ with $G(r_1) = 0^{(3)}$.
        By the same M\"obius transformation $\phi_1$ and notation as (2), 
        we set $\ti G \coloneqq G \circ \phi_1^{-1}$, and this 
        $\ti G$ can be written as 
        $$
          \ti G = C \fr{(z - 1)^3 (z - a)}{z^2} \qquad (C \neq 0).
        $$
        Let us denote $\xi$ and $\eta$ by the roots of the polynomial 
        $P_1(z) \coloneqq 2 z^2 + (1 - a) z - 2 a$ (see \eqref{eq:tiG_(2)}).
        Then, since they are double roots of $P_1$ and 
        $P_2(z) \coloneqq C (z - 1)^3 (z - a) - z^2$, which is the numerator of $\ti G$, 
        by applying Vieta's formula to $P_2$ and $P_1$, 
        we have $\xi + \xi + \eta + \eta = 2 (\xi + \eta) = a + 3$ and  
        $\xi + \eta = (1/2)(a - 1)$, respectively.
        These two equations yield $a - 1 = a + 3$, 
        which is a contradiction, and we get the assertion. 
        
      \item
        We set $G^{-1}(0) = \{p_1, p_2\} \subset E$, which satisfies $G(p_1) = G(p_2) = 0^{(2)}$.
        Take a M\"obius transformation $\phi_1 : \Chat \to \Chat$ to satisfy 
        $\phi_1(q_1) = 0, \phi_1(q_2) = \infty, \phi_1(p_1) = r_1$ for some $r_1 \neq 0$, 
        and let $r_2 \coloneqq \phi_1(p_2)$.
        Then, we can write $(G \circ \phi_1^{-1}) = H(z)^2$, where 
        $$
          H(z) \coloneqq \rt C \ \fr{(z - r_1) (z - r_2)}{z}  \qquad (C \neq 0).
        $$
        Setting $(G \circ \phi_1^{-1})^{-1}(1) = \{y_1, y_2\} \subset \phi_1(E)$, 
        we know $(G \circ \phi_1^{-1})(y_j) = H(y_j)^2 = 1$ and $(G \circ \phi_1^{-1})'(y_j) = 0$ $(j = 1, 2).$
        Since $(G \circ \phi_1^{-1}) ' (y_j) = 2 H(y_j) H'(y_j)$, $H(y_j) \neq 0$,  and $H'(z) = \rt c (z^2 - r_1 r_2) / z^2$, 
        we have $\{y_1, y_2\} = \{\pm \rt{r_1 r_2}\}$.
        By $H(\pm \rt{r_1 r_2}) = \pm \rt c (2 \rt{r_1 r_2} \mp (r_1 + r_2))$, where the signs correspond, 
        and $H(\pm \rt{r_1 r_2}) \in \{\pm 1\}$, it holds that 
        $H(\rt{r_1 r_2}) + H(- \rt{r_1 r_2}) = -2 \rt c (r_1 + r_2) = 0$.
        Thus, one see $r_2 = -r_1$, and hence 
        $$
          G \circ \phi_1^{-1} = \left(\rt C \ \fr{z^2 - r_1^2}{z} \right)^2.
        $$
        Through changing coordinate $z = \psi_1(w) = r_1 w$, we set 
        $$
          \ti G (w) \coloneqq (G \circ \phi_1^{-1})(\psi_1(w)) = K \fr{(w^2 - 1)^2}{w^2} 
          \qquad (K \coloneqq C r_1^4).
        $$  
        
       The polynomial $K (w^2 - 1)^2 - w^2 = K w^4 - (2 K + 1) w^2 + K$, 
       which is the numerator of $\ti G - 1$, 
        has two non-zero double roots.
        Set $w^2 = X$, and then $K X^2 - (2 K + 1) X + K$ must have 
        one non-zero double root.
        The discriminant $(2K + 1)^2 - 4 K ^2 = 4 K + 1$ of 
        this polynomial must be $0$, namely $K = -1 / 4$. 
        Thus, $\ti G$ can be written as 
        $$
          \ti G = - \fr{(w^2 - 1)^2}{4 w^2}, 
        $$
        which satisfies $\ti G^{-1}(1) = \{\pm i\}$,
        and hence $(\psi_1^{-1} \circ \phi_1)^{-1} (E) = \{\pm 1, \pm i\}$.
        Direct computation gives $[1, -1 ; i, - i] = -1 = [\infty, 0 ; i, -i]$, and these arguments so far imply that 
        there exists a biholomorphic map $\rho : \Chat \to  \Chat$ such that $\rho(E) = \{\infty, \pm i, 0\}$,
        which yields the assertion.
    \end{enumerate}
  \end{proof}
  
   Let us consider {\bf Case 1} and determine the form of the meromorphic function 
   $G : \Chat \to \Chat$. 
   First, we prepare a lemma.
  
  \begin{lemma} \label{lem:cross_ratio_1}
    Let $\{p_0, p_1, s_1, s_2\} \subset\Chat$ be distinct points.
    Assume that there exists a meromorphic function $G : \Chat \to \Chat$ such that 
    $\deg G = 4$, and it satisfies the following conditions:
    \begin{enumerate}[label = (\arabic*), leftmargin = *, font = \normalfont]
      \item 
        $G^{-1}(0) = \{p_0\}$, and $G^{-1}(1) = \{p_1, s_1, s_2\}$,
      \item
        $G(p_0) = 0^{(4)}, G(p_1) = 1^{(2)},$ and $G(s_1) = G(s_2) = 1^{(1)}$,
      \item
        the value $\infty$ is the unique totally ramified value, and $G(q_1) = G(q_2) = \infty^{(2)}$, 
        where $G^{-1}(\infty) = \{q_1, q_2\}$.
   \end{enumerate}
   Then, it must hold that $[p_0, p_1 ; s_1, s_2] = - 1$.
  \end{lemma}
  
  \begin{proof}
    Let $\phi : \Chat \to \Chat$ be a M\"obius transformation satisfying 
    $\phi(p_0) = \infty$ and $\phi(p_1) = 0$, and set $\ti G \coloneqq G \circ \phi^{-1}$.
    Let $\ti s_1 \coloneqq \phi(s_1), \ti s_2 \coloneqq \phi(s_2), \al \coloneqq \phi(q_1),$ and 
    $\bt \coloneqq \phi(q_2)\ (\al \neq \bt).$
    We only have to show 
    $[p_0, p_1 ; s_1, s_2]= [\infty, 0 ;  \ti s_1, \ti s_2] = -1$, and 
    so we can set 
    $$
      \ti G = \fr{1}{P(z)} \qquad (P(z) \coloneqq C (z - \al)^2 (z - \bt )^2, \quad C \neq 0).
    $$
    Since $(\ti G - 1)' (0) = \ti G' (0) = 0$, by
    $$
      P'(z) = \left( \fr{2}{z - \al} + \fr{2}{z - \bt} \right) P(z),
    $$ 
    we have $\bt = -\al$, and $\ti G (0) = 1$ yields $C = 1 / \al^4$.
    So, we can set $P(z) = (k z^2 - 1)^2$, where $k \coloneqq 1 / \al^2$.
    One can observe $\{\ti s_1, \ti s_2\} = \{\xi, - \xi\}$ for some $\xi \neq 0$ 
    by solving an equation $\ti G = 1$.
    Hence, we get $[\infty, 0 ; \ti s_1, \ti s_2] = -1$, which completes the proof.
  \end{proof}
  
  Let us construct the meromorphic function $G$.
  Without loss of generality, we may set $\Sigma = \Chat \setminus \{\infty, \pm i, 0\}$
   by \ref{sphere_structure_1} in Lemma \ref{lem:structure_of_source} and suppose $G(\infty) = 0^{(4)}$
  by  coordinate changes fixing the set $\{\infty, \pm i, 0\}$.
  Then, the possible values and multiplicities of $G$ at the other ends are as follows:
  (i) $G(0) = 1^{(2)}$, and $G(i) = G(-i) = 1^{(1)}$ or (ii) $G(i) = 1^{(2)}$, and $G(- i) = G(0) = 1^{(1)}$ or
  (iii) $G(- i) = 1^{(2)}$, and $G(0) = G(i) = 1^{(1)}$.
  On the other hand, calculating the cross ratio in each case of (i) through (iii), 
  we know (i) $[\infty, 0 ; i, -i] = -1$, (ii) $[\infty, i; -i, 0] = 1/2$, and (iii) $[\infty, -i; 0, i] = 2$.
  From Lemma \ref{lem:cross_ratio_1}, the cases (ii) and (iii) are impossible.
  It follows that we only consider the case (i).
  
  Recalling \eqref{eq:asum_G} and the arguments there, we can set 
   \begin{equation} \label{eq:G_setting}
      	G(z) = \fr{N(z)}{P(z)^2},
   \end{equation}
   where $N$ and $P$ are polynomial functions.
   We denote 
   $$
   	P(z) \coloneqq z^2 - s_1 z + s_2 \qquad (s_1 \coloneqq q_1 + q_2, \quad s_2 \coloneqq q_1 q_2),
   $$
   which satisfies $P(q_j) = 0$, $P'(q_j) \neq 0\ (j = 1, 2)$, and $P'' \equiv 2$.
   In the sequel, we are to determine $N, s_1$, and $s_2$.
   
   Since $G(\infty) = 0^{(4)}$, we can set 
	$$
	   G = \fr{\lambda}{P(z)^2} \qquad (\lambda \in \C \setminus \{0\}).
	$$ 
	The condition $(G - 1)(0) = (G - 1)(\pm i) = 0$ and $G(0) = 1^{(2)}$ imply
	\begin{align*}
		&(G - 1)(0) = \fr{\lambda - P(0)^2}{P(0)^2} = 0,
		\qquad
		(G - 1)(\pm i) =  \fr{\lambda -  P(\pm i)^2}{P(\pm i)^2} = 0,\\
		&(G - 1)'(0) = - \fr{2 \lambda P'(0)}{P(0)^3} = 0.
	\end{align*}
	The third equation yields 
	$
		s_1 = 0.
	$
	In turn, the first and second equation give $\lambda = P(0)^2 = s_2^2 = P(\pm i)^2$.
	Note that $P(i) = - 1 + s_2 = P(-i)$.
	Thus, $P(0)^2 = P(i)^2$ gives $s_2 = 1/2$ and $\lambda = 1/4$.
	Accordingly, we obtain 
	\begin{equation} \label{eq:G_1-1-1}
	   G = \fr{1}{(2z^2 +1)^2}. 
	\end{equation}
  
   Next, let us investigate {\bf Case 2} by the same way as {\bf Case 1}.
  \begin{lemma} \label{lem:cross_ratio_2}
    Let $\{p_3, p_1, s_3, s_1\} \subset\Chat$ be distinct points,
    and assume that a meromorphic function $G : \Chat \to \Chat$ satisfy 
    $\deg G = 4$, and the following conditions:
    \begin{enumerate}[label = (\arabic*), leftmargin = *, font = \normalfont]
      \item 
        $G^{-1}(0) = \{p_3, p_1\}$, and $G^{-1}(1) = \{s_3, s_1\}$,
      \item
        $G(p_3) = 0^{(3)}, G(p_1) = 0^{(1)}, G(s_3) = 1^{(3)}, $ and $G(s_1) = 1^{(1)}$,
      \item
        the value $\infty$ is the unique totally ramified value, and $G(q_1) = G(q_2) = \infty^{(2)}$, 
        where $G^{-1}(\infty) = \{q_1, q_2\}$.
   \end{enumerate}
   Then, it must hold that $[p_3, p_1 ; s_3, s_1] = 9$.
  \end{lemma}
  
  \begin{proof}
    Let $\phi : \Chat \to \Chat$ be a M\"obius transformation satisfying 
    $\phi(q_1) = 0, \phi(q_2) = \infty$, and $\phi(p_3) = 1$,  and set $\ti G \coloneqq G \circ \phi^{-1}$.
    Let $a \coloneqq \phi(p_1), w_3 \coloneqq \phi(s_3),$ and $w_1 \coloneqq \phi(s_1)$.
    Since $[p_3, p_1 ; s_3, s_1] = [1, a ;  w_3, w_1]$, we only have to show 
    $[1, a ;  w_3, w_1] = 9$. 
    By assumption, we can set 
    $$
      \ti G = K \fr{(z - 1)^3 (z - a)}{z^2} \qquad (K\neq 0).
    $$
    The condition $\ti G(w_3) = 1^{(3)}$ yields $(\ti G - 1)'(w_3) = \ti G' (w_3) = 0^{(2)}$.
    Since
    $$
      \ti G' = \fr{2z^2 + (1 - a) z - 2 a}{z (z - 1) (z - a)} \,  \ti G,
    $$
    the point $w_3$ is the double root of the polynomial $2z^2 + (1 - a) z - 2 a$ 
    because of $\ti G(w_3) \neq 0$.
    Thus, the points $w_3$ and $a$ satisfy $(1 - a)^2 - 4 \cdot 2 \cdot (- 2 a) = a^2 + 14 a  + 1 = 0$
    and $w_3 = (a - 1) / 4$.
    
    Additionally, applying Vieta's formula to the polynomial $K(z - 1)^3(z - a) - z^2$, 
    which is the numerator of $\ti G - 1$, we get $w_1 = (a + 15) / 4.$ 
    As a consequence, by noting $a^2 = - 14 a - 1$, we obtain the assertion as follows:
    $$
      [p_3, p_1 ; s_3, s_1]  = [1, a ; w_3, w_1] 
       = \fr{3(a^2 - 10 a + 25)}{3 a^2 + 34 a +11} = \fr{72 (1 - a)}{8 (1 - a)} = 9. 
    $$    
  \end{proof}
  
  Let us determine the form of the meromorphic function $G$ satisfying the condition in
  {\bf Case 2}.
  We may set $\Sigma = \Chat \setminus \{\infty, \pm i, - 17 i\}$
   by \ref{sphere_structure_2} in Lemma \ref{lem:structure_of_source} and suppose $G(\infty) = 0^{(3)}$.
  Then, the possible values and multiplicities of $G$ at the other ends are following $6$ cases:
   \begin{align*}
   & \ \quad  \text{(i)}& G(i) = 0^{(1)}, \quad G(- i) = 1^{(3)},  \quad &\text{ and } \quad G(- 17 i) = 1^{(1)} \\
   & \text{or (ii)} & G(i) = 0^{(1)},\quad  G(- 17 i) = 1^{(3)}, \quad &\text{ and } \quad G(- i) = 1^{(1)} \\
   & \text{or (iii)} & G(- i) = 0^{(1)}, \quad G(i) = 1^{(3)}, \quad &\text{ and } \quad G(-17 i) = 1^{(1)} \\
   & \text{or (iv)} & G(- i) = 0^{(1)}, \quad G(- 17 i) = 1^{(3)}, \quad &\text{ and } \quad G(i) = 1^{(1)} \\
   & \text{or (v)} & G(- 17 i) = 0^{(1)}, \quad G(i) = 1^{(3)}, \quad &\text{ and } \quad G(- i) = 1^{(1)} \\
   & \text{or (vi)} & G(- 17 i) = 0^{(1)}, \quad G(- i) = 1^{(3)}, \quad &\text{ and } \quad G(i) = 1^{(1)}.
   \end{align*}
  The cross ratios in each case of (i) through (vi) are
  (i) $[\infty, i ; - i, - 17 i] = 9$, (ii) $[\infty, i; - 17 i, - i] = 1 / 9$, (iii) $[\infty, -i; i, -17 i] = - 8$,
  (iv) $[\infty, -i; -17 i, i] = -1 / 8$, (v) $[\infty, -17 i; i, - i] = 8 / 9$, and (vi) $[\infty, - 17 i; - i, i] = 9 / 8$.
  By Lemma \ref{lem:cross_ratio_2}, the cases (ii) through (vi) are impossible.
  Thus, we only have to consider the case (i).
  
  We use the same notation as \eqref{eq:G_setting}, and let
   $$
   	F(z) \coloneqq N(z) - P(z)^2.
   $$
   Since $G(\infty) = 0^{(3)}$ and $G(i) = 0^{(1)}$, we can set
    $$
      G = \fr{\lambda (z - i)}{P(z)^2} \qquad (\lambda \neq 0).
   $$
    Also, since $G(- i) =  1^{(3)}$, we obtain $F(-i) = F'(-i) = F''(-i) = 0$, and these are equivalent to 
   $$
       P(-i)^2 = -2i \lambda, \qquad \lambda = 2P(-i) P'(-i), \qquad
       2P(-i) = -P'(-i)^2. 
   $$
   These equations imply 
   $$
     \lambda = 512 i, \qquad s_1 = -10 i, \qquad s_2 = 23,
   $$
   and thus we get 
   \begin{equation} \label{eq:G_2-1-1}
	G = \fr{512i(z - i)}{(z^2 + 10iz + 23)^2},
    \end{equation}
    which satisfies $G(- 17 i) = 1^{(1)}$.

   Finally, let us consider {\bf Case 4}.
      
  \begin{lemma} \label{lem:cross_ratio_4}
    Let $\{p_1, p_2, s_1, s_2\} \subset\Chat$ be distinct points, and assume that
    a meromorphic function $G : \Chat \to \Chat$ satisfy
    $\deg G = 4$, and the following conditions:
    \begin{enumerate}[label = (\arabic*), leftmargin = *, font = \normalfont]
      \item 
        $G^{-1}(0) = \{p_1, p_2\}$, and $G^{-1}(1) = \{s_1, s_2\}$,
      \item
        $G(p_1) = G(p_2) =  0^{(2)},$ and  $G(s_1) = G(s_2) = 1^{(2)}$,
      \item
        the value $\infty$ is the unique totally ramified value, and $G(q_1) = G(q_2) = \infty^{(2)}$, 
        where $G^{-1}(\infty) = \{q_1, q_2\}$.
   \end{enumerate}
   Then, it must hold that $[p_1, p_2 ; s_1, s_2] = -1$.
  \end{lemma}
  
  \begin{proof}
    Let $\phi : \Chat \to \Chat$ be a M\"obius transformation satisfying 
    $\phi(q_1) = 0, \phi(q_2) = \infty$, and $\phi(p_1) = 1$,  
    and set $\ti G \coloneqq G \circ \phi^{-1}$.
    Let $a \coloneqq \phi(p_2), w_1 \coloneqq \phi(s_1),$ and $w_2 \coloneqq \phi(s_2)$.
    Since $[p_1, p_2 ; s_1, s_2] = [1, a ;  w_1, w_2]$, let us show 
    $[1, a ;  w_1, w_2] = -1$. 
    We may set 
    $$
      \ti G = K \fr{(z - 1)^2 (z - a)^2}{z^2} = K \left(\fr{(z - 1) (z - a)}{z}\right)^2 \qquad (K\neq 0).
    $$
    Since $\ti G(w_j) = 1^{(2)}$ ($j = 1, 2$), and thus $(\ti G - 1)'(w_j) = \ti G' (w_j) = 0,$
    direct computation gives 
    $\{w_1, w_2\} = \{\rt a, - \rt a\}$.
    The condition $\ti G(w_j) = 1$ yields $- K (\rt a - 1)^4 = 1$ and $K (\rt a + 1)^4 = 1$,
    and hence $a = -1$. 
    As a result, we have $\{w_1, w_2\} = \{i, -i\}$.
    Therefore, we get the assertion as follows:
    $$
      [1, a; w_1, w_2] = [1, -1; i, -i] = [1, -1; -i, i] = -1.
    $$
  \end{proof}
  Thus, we may set $\Sigma = \Chat \setminus \{\infty, \pm i, 0\}$ in {\bf Case 4}
   by \ref{sphere_structure_4} in Lemma \ref{lem:structure_of_source} and suppose $G(\infty) = 0^{(2)}$.
  Then, the possible values and multiplicities of $G$ at the other ends are following $3$ cases:
  (i) $G(0) = 0^{(2)}$, and $G(i) = G(- i) = 1^{(2)}$ or 
  (ii) $G(i) = 0^{(2)}$, and $G(- i) = G(0) = 1^{(2)}$ or 
  (iii) $G(- i) = 0^{(2)}$, and $G(0) = G(i) = 1^{(2)}$.
  In each case of (i) through (iii), the cross ratio can be computed as follows, as in Lemma \ref{lem:cross_ratio_4}:
  (i) $[\infty, 0 ; i, -i] = -1$, (ii) $[\infty, i ; - i, 0] = 1 / 2$, (iii) $[\infty, - i ; 0, i] = 2$.
  Hence, the cases (ii) and (iii) are impossible.
  
  Next, let us construct the meromorphic function $G$ satisfying the condition of (i).
  Since $G(\infty) = G(0) = 0^{(2)}$, by using the same notation as \eqref{eq:G_setting}, we can set 
  $$
    G = \fr{\lambda\, z^2}{P(z)^2} \qquad (\lambda \in \C \setminus \{0\}). 
  $$
  By $G(i) =  G(- i) = 1^{(2)}$, $h(\pm i) = h'(\pm i)  = 0$ holds, where $h(z) \coloneqq z^2 - P(z)^2$,
  and it implies
  $$
    P(i)^2 =  - \lambda, \quad P(0) P'(0) = - i \lambda , \quad 
   P(-i)^2 =  - \lambda, \quad P(-i) P'(-i) = - i \lambda.
  $$ 
   Noting $\lambda \neq 0$, we have 
   $\lambda = - 4, s_1 = 0,$ and $s_2 = - 1, $
   which give 
   \begin{equation} \label{eq:G_4-2}
      G = -\fr{4 z^2}{(z^2 - 1)^2}. 
    \end{equation}
 
  Summing up the argument so far, the following lemma is obtained.
 \begin{lemma} \label{lem:construction_G}
   Let $G : \Sigma = \Chat \setminus \{\infty, \pm i, t\} \to \Chat$ be a meromorphic function.
   Suppose that  
   it omits $0$ and $1$, has totally ramified value $\infty$ other than omitted values,
   can be extended meromorphically to $\Chat$, and $\deg G = 4$.
   Then, its $G : \Sigma \to \Chat$ can be determined one of \eqref{eq:G_1-1-1} defined on 
   $\Chat \setminus \{\infty, \pm i, 0\}$, \eqref{eq:G_2-1-1} defined on $\Chat \setminus \{\infty, \pm i, - 17 i\}$, 
   or \eqref{eq:G_4-2} defined on $\Chat \setminus \{\infty, \pm i, 0\}$.
 \end{lemma}
 
  With the preparations made so far, 
  we see the existence complete minimal immersions defined on 
  the four-punctured sphere with
  $C(\Sigma) = -16 \pi, D_g = 2, $ and $\nu_g = 2.5$
  other than Kawakami--Watanabe's example as follows.
  
  \begin{theorem} \label{thm:TC16Pi_25}
    Any complete minimal immersion 
    $\Sigma \coloneqq \Chat \setminus \{\text{$4$ points}\} \to \R^3$ 
    of total curvature $-16 \pi$ with $D_g = 2$ and $\nu_g = 2.5$
    is given by the 
    following Weierstrass data $(g, \omega)$ 
    defined on $\Chat \setminus \{\infty, \pm i, 0\}$:
    for distinct values $\sigma \in \C$, $\tau, b \in \C$
     and a value $\theta \in \C \setminus \{0\}$,
    
    \noindent
    {\bf Case 1}
    		\begin{equation}
		  g =  \frac{\sigma (b - \tau)(2z^2 + 1)^2 + b (\tau - \sigma)}{(b - \tau)(2z^2 + 1)^2 + (\tau - \sigma)},
		\end{equation}
		and the holomorphic $1$-form $\omega$ on $\Sigma$ meromorphic on $\Chat$ 
		is given by either
		 
		\begin{equation} 
		  (1)\qquad 
		  \omega = \theta \dfrac{((b - \tau)(2z^2 + 1)^2 + (\tau - \sigma))^2}
		  {z^2 (z - i)^2 (z + i)^2} dz,
		\end{equation}
		with
		\begin{equation} \label{eq:per_Dg2_1_(1)}
		  \left\{
		  \begin{array}{l}
		    \Re \left( \theta (b - \sigma) 
		    (b(16 \sigma \tau - 3 \tau^2 - 13) - \sigma (13 \tau^2 + 3) + 16\tau)\right) = 0,\\
		   \Im\left( \theta (b - \sigma) 
		   ( b(16 \sigma \tau - 3 \tau^2 + 13) - \sigma (13 \tau^2 - 3) - 16\tau) \right) = 0, \\ 
		   \Re\left(  \theta (b - \sigma) 
		   (b(8 \sigma + 5  \tau) - \tau(5 \sigma + 8\tau)) \right) = 0\\ 
		  \end{array}
		  \right.
		\end{equation} 
		{\nf (}e.g. $\tau = 0, \ b = - (3 / 13) \sigma, \  \sigma^3 \in i\R,  \ \theta = 1${\nf )} or
		
		\begin{equation}
		   (2) \qquad
		     \omega = \theta \dfrac{((b - \tau)(2z^2 + 1)^2 + (\tau - \sigma))^2}
		     {z^4 (z - i)^2 (z + i)^2} dz,
		\end{equation}
		with 
		\begin{equation} \label{eq:per_Dg2_1_(2)}
		  \left\{
		  \begin{array}{l}
		   \Re \left
		   ( \theta (b - \sigma) (b(16 \sigma \tau - 5 \tau^2 - 11) - \sigma (11 \tau^2 + 5) + 16 \tau) 
		   \right) = 0, \\ 
		  \Im\left(
		   \theta (b - \sigma) (b(16 \sigma \tau - 5 \tau^2 + 11) - \sigma (11 \tau^2 - 5) - 16 \tau) 
		   \right) = 0, \\
  		  \Re \left(
  		     \theta (b - \sigma) (b(8\sigma + 3 \tau) - \tau(3 \sigma + 8\tau) )
  		    \right) = 0 \\
		  \end{array}
		  \right.
		\end{equation}
		{\nf (}e.g. $\tau = 0, \ b = - (5 / 11) \sigma, \ \sigma^3 \in i \R,\  \theta = 1${\nf)}, or
		
		\begin{equation}
		(5) \qquad
		       \omega = \theta \dfrac{((b - \tau)(2z^2 + 1)^2 + (\tau - \sigma))^2}{z^2 (z - i)^3 (z + i)^3} dz
		\end{equation}
		with
		\begin{equation} \label{eq:per_Dg2_1_(5)}
		  \left\{
		  \begin{aligned}
     		   &\Re\left(
     		     \theta  (b^2  (128 \sigma^2 - 32 \sigma \tau + 15 \tau ^2 - 111) 
     		    - 2 b (112 \sigma ^2 \tau - \sigma  \tau ^2 + \sigma - 112 \tau ) \right.\\
     		  & \left.
     		  \hspace{3cm} + 3 \sigma ^2 (37 \tau ^2 - 5)+32 \tau(\sigma  - 4 \tau))
     		  \right) = 0, \\ 
     		 &\Im \left(
     		  \theta  (b^2 (128 \sigma ^2 - 32 \sigma \tau + 15 \tau ^2 + 111) 
     		  + 2 b (-112 \sigma ^2 \tau +  \sigma  \tau ^2 + \sigma - 112 \tau )) \right.\\
     		& \left.
     		\hspace{3cm} + 3 \sigma ^2 (37 \tau ^2 + 5) - 32 \tau(\sigma  + 4 \tau))
     		\right) = 0, \\
      		&\Re \left(
      		    \theta  (b^2 (112 \sigma - \tau ) + 2 b (8 \sigma ^2 - 127 \sigma  \tau + 8 \tau ^2)
      		   - \sigma \tau  (\sigma - 112 \tau )) 
      		   \right) = 0 \\
		  \end{aligned}
		  \right.
		\end{equation}
		{\nf(}e.g. $\tau = 0, \ b = - (1/7) \sigma , \ \sigma = \rt{13 / 2}, \ \theta = 1${\nf)} or

		\begin{equation} 
		    (8) \qquad 
		    \omega = \theta \dfrac{((b - \tau)(2z^2 + 1)^2 + (\tau - \sigma))^2}{z^3 (z - i)^2 (z + i)^2} dz
		\end{equation}
		with
		\begin{equation} \label{eq:per_Dg2_1_(8)}
		  \left\{
		    \begin{array}{l}
    			\Im\left(
    			  \theta  (b - \sigma ) (b (4 \sigma \tau - \tau^2 - 3) - \sigma (3 \tau^2 + 1) + 4 \tau)
    			  \right) = 0, \\
    			  \Re \left(\theta  (b - \sigma) (b (4 \sigma \tau - \tau^2 + 3) - \sigma(3 \tau^2 - 1)- 4 \tau)
    			  \right) = 0, \\
    			  \Im\left(
    			  \theta  (b-\sigma ) (b (2 \sigma +\tau ) -\tau  (\sigma + 2 \tau ))
    			  \right) = 0 
		    \end{array}
		  \right.
		  \end{equation}
		{\nf(}e.g. $\tau = 0, \ b = - (1/3)\sigma, \ \sigma^3 \in \R, \ \theta = 1${\nf)} or
		
		\noindent
		{\bf Case 4} 
		\begin{equation}
		  g =  \frac{\sigma (b - \tau)(z^2 - 1)^2 + 4 b (\sigma - \tau) z^2}
    				{(b - \tau)(z^2 - 1)^2 + 4 (\sigma - \tau) z^2}
		\end{equation}
		and the holomorphic $1$-form $\omega$ on $\Sigma$ meromorphic on $\Chat$ 
		is given by either
		\begin{equation} 
		 (5) \qquad
		   \omega = \theta \dfrac{((b - \tau)(z^2 - 1)^2 + 4 (\sigma - \tau) z^2)^2}
    				{z^2 (z - i)^3 (z + i)^3} dz
		\end{equation}
		with
		\begin{equation} \label{eq:per_Dg2_4_(5)}
		  \left\{
		    \begin{array}{l}
		      	\Re \left(
		      	  \theta (b - \sigma ) (b (4 \sigma\tau - \tau^2 - 3) - \sigma (3 \tau^2 + 1) + 4 \tau )
		      	 \right) = 0,\\
			\Im \left(
			 \theta (b - \sigma) (b (4 \sigma \tau - \tau^2 + 3) - 3 \sigma \tau^2 + \sigma - 4 \tau )
			 \right) = 0, \\ 
			 \Re \left(
			   \theta (b - \sigma) (b (2 \sigma + \tau ) - \tau (\sigma + 2 \tau ))
			  \right) = 0
		    \end{array}
		  \right.
		\end{equation}
		{\nf(}e.g. $\tau = 0, \ b = - (1 / 3) \sigma, \ \sigma^3 \in i\R, \ \theta = 1${\nf)}. 
	
  \end{theorem}
  
  \begin{proof}
    By using the meromorphic function $G$ constructed in each case
    in Lemma \ref{lem:construction_G}, one can determine the form of 
    $g = \Phi^{-1} \circ G$, 
    where $\Phi$ is as in \eqref{eq:Mobius_2}.
    Also in general (cf. \cite{Osserman}), 
    by the regularity condition \eqref{eq:1stff}, 
    since $g$ has  $4$ simple poles at points in $\Chat \setminus  \{\infty, \pm i, t\}$, 
    then $\omega$ has zero points there of order $2$.
    Since $g$ has no pole at each end, $\omega$ has poles there of order greater 
    than or equal to $2$, and it is holomorphic at the other points in 
    $\Chat \setminus \{\infty, \pm i, t\}$.
    Thus, Riemann--Roch's theorem yields 
    $$
    	\sum_{p \in \omega^{-1} (\infty)} \ord_p \omega = -10, 
    $$
    where $\omega^{-1}(\infty)$ denotes the set of poles of $\omega$ in $\Chat$.
    Hence, if $g$ is written as $g = g_2 / g_1$, then, one can set $\omega$ as follows:
    \begin{equation}\label{eq:omega_form}
    \begin{aligned}
	  &(1) \ 
	    \omega = \theta \fr{g_1(z) ^2}{(z - t)^2 (z - i)^2 (z + i)^2}dz,
	  \quad  
	  &(2) \ 
	    \omega = \theta \fr{g_1(z) ^2}{(z - t)^4 (z - i)^2 (z + i)^2}dz, \\
	  &(3) \  
	    \omega = \theta \fr{g_1(z) ^2}{(z - t)^2 (z - i)^4 (z + i)^2}dz,
	  \quad 
	  &(4) \ 
	    \omega = \theta \fr{g_1(z) ^2}{(z - t)^2 (z - i)^2 (z + i)^4}dz, \\ 
	  &(5) \ 
	    \omega = \theta \fr{g_1(z) ^2}{(z - t)^2 (z - i)^3 (z + i)^3}dz, 
	  \quad 
	  &(6) \ 
	    \omega = \theta \fr{g_1(z) ^2}{(z - t)^3 (z - i)^3 (z + i)^2}dz, \\   
	  &(7) \ 
	   \omega = \theta \fr{g_1(z) ^2}{(z - t)^3 (z - i)^2 (z + i)^3}dz, 
	  \quad 
	  &(8) \ 
	    \omega = \theta \fr{g_1(z) ^2}{(z - t)^3 (z - i)^2 (z + i)^2}dz, \\ 
	  &(9) \ 
	    \omega = \theta \fr{g_1(z) ^2}{(z - t)^2 (z - i)^3 (z + i)^2}dz,  
	  \quad 
	  &(10) \ 
	    \omega = \theta \fr{g_1(z) ^2}{(z - t)^2 (z - i)^2 (z + i)^3}dz. \\ 
     \end{aligned}
     \end{equation}
       
     \noindent
     Here,  $g_1$ and $g_2$ are coprime polynomials of degree $4$, and $\theta \in \C \setminus \{0\}$.
    Hence, 
    we will deduce whether the pair $(g, \omega)$ in each case 
    satisfies the period condition \eqref{eq:PC}.
    Since the source Riemann surface is $\Chat \setminus \{\infty, \pm i, t\}$, 
    \eqref{eq:PC} is equivalent to 
    \begin{equation} \label{eq:period_zero}
      \Res(\al , p) \in \R
    \end{equation}
    for all $p \in \{\pm i, t\}$, 
    where the $\C^3$-valued meromorphic $1$-form $\al$ is defined by
    \begin{equation} \label{W-1form}
    	\al \coloneqq (\al_1, \al_2, \al_3) \coloneqq \fr{1}{2} (1 - g^2, i (1 + g^2), 2g) \omega,
    \end{equation}
    and $\Res(\al, p)$ denotes the residue of $\al$ at $p$.

  Let us first investigate {\bf Case 1}.
  \eqref{eq:G_1-1-1} and \eqref{eq:Mobius_2} give
  $$
    	g =  \dfrac{\sigma (b - \tau)(2z^2 + 1)^2 + b (\tau - \sigma)}
    	{(b - \tau)(2z^2 + 1)^2 + (\tau - \sigma)}.
   $$
  By considering a symmetry $z \mapsto -z$, 
  we only have to check the period problem other than 
  (4), (7), (10) in \eqref{eq:omega_form}.    

  \noindent
     	(1)   The meromorphic $1$-form $\omega$ is given by
  	$$
         \omega = \theta \dfrac{((b - \tau)(2z^2 + 1)^2 + (\tau - \sigma))^2}{z^2 (z - i)^2 (z + i)^2} dz,
	$$
  Then, we get
  \begin{align*}
	\Res(\al_1, i) &= \fr{i}{8}  \theta (b - \sigma) (b(16 \sigma \tau - 3 \tau^2 - 13) - \sigma (13 \tau^2 + 3) + 16\tau)
	= -\Res(\al_1, -i)\\
	\Res(\al_2, i) &= \fr{1}{8}  \theta (b - \sigma) (b(16 \sigma \tau - 3 \tau^2 + 13) - \sigma (13 \tau^2 - 3) - 16\tau)
	= -\Res(\al_2, -i)\\
	\Res(\al_3, i) &= - \fr{i}{4}  \theta (b - \sigma) (b(8\sigma + 5\tau) -  \tau(5 \sigma + 8\tau)) 
	= - \Res(\al_3, -i),\\
	\Res(\al, 0) &= \bm 0,
  \end{align*}
  where $\bm 0 \coloneqq (0,0,0)$.
  Hence, the period condition is equivalent to \eqref{eq:per_Dg2_1_(1)}, and 
  one can see that there exist $b, \sigma, \tau$, and $\theta$ that satisfy the condition.

  \noindent
 (2) 
 The meromorphic $1$-form $\omega$ is expressed as 
 $$
   \omega = \theta \dfrac{((b - \tau)(2z^2 + 1)^2 + (\tau - \sigma))^2}{z^4 (z - i)^2 (z + i)^2} dz.
 $$
 Then, 
 \begin{align*}
  \Res(\al_1, i) &= \fr{i}{8} \theta (b - \sigma) ( b(16 \sigma \tau - 5 \tau^2 - 11) - \sigma (11 \tau^2 + 5) + 16 \tau) 
			= -\Res(\al_1, -i),\\
  \Res(\al_2, i) &= \fr{1}{8}  \theta (b - \sigma) ( b(16 \sigma \tau - 5 \tau^2 + 11) - \sigma (11 \tau^2 - 5) - 16 \tau) 
  = -\Res(\al_2, -i),\\
  \Res(\al_3, i) &= - \fr{i}{4}  \theta (b - \sigma) (b(8\sigma + 3 \tau) - \tau(3 \sigma + 8\tau)) 
  = - \Res(\al_3, -i),\\
  \Res(\al, 0) &= \bm 0.
 \end{align*}
 Hence, the period condition is equivalent to \eqref{eq:per_Dg2_1_(2)}, and 
  one can see that there exist $b, \sigma, \tau$, and $\theta$ that satisfy the condition.
  
 \noindent
 (3) 
  The meromorphic $1$-form $\omega$ is expressed as 
  \begin{equation}
   \omega = \theta \dfrac{((b - \tau)(2z^2 + 1)^2 + (\tau - \sigma))^2}{z^2 (z - i)^4 (z + i)^2} dz.
  \end{equation}
 Then, 
 \begin{align*}
  \Res(\al_1, 0) &= - i \theta (b - \sigma)^2 (\tau^2 - 1),
  \qquad 
  \Res(\al_2, 0) = -\theta (b - \sigma)^2 (\tau^2 + 1), \\
  \Res(\al_3, 0) &= 2 i \theta (b - \sigma)^2 \tau.
 \end{align*}
 Then, the period condition \eqref{eq:period_zero} gives
       \begin{equation} \label{eq:Dg_2_period_no}
          (\tau^2 - 1) U \in i\R, \qquad 
          (\tau^2 + 1) U \in \R, \qquad 
          \tau U \in i \R,
       \end{equation}
       where $U \coloneqq \theta (b - \sigma)^2 \neq 0$.
       If $\tau ^2 + 1 \neq 0$, 
       then the first and second conditions in \eqref{eq:Dg_2_period_no} yield
	$$
		\fr{\tau^2 - 1}{\tau^2 + 1} \in i \R.
	$$
	This is equivalent to $|\tau| = 1$, and so we can set $\tau = e^{i \phi}\ (\phi \not \equiv \pi/2 \mod \pi)$.
	Since $\tau^2 + 1 = 2 e^{i\phi} \cos \phi$, the condition $ (\tau^2 + 1) U \in \R$ gives 
	$e^{i \phi} U \in \R$.
	On the other hand, the third condition means $e^{i\phi} U \in i \R$.
	Hence, it holds that $e^{i \phi} U = 0$, which contradicts $U \neq 0$.
	
	Also if $\tau^2 + 1 = 0$, then the first and third conditions in \eqref{eq:Dg_2_period_no} show that
	$U \in i \R \cap \R$, which implies $U = 0$
	and a contradiction.
	
 \noindent
 (5) 
  The meromorphic $1$-form $\omega$ is expressed as 
  \begin{equation}
    \omega = \theta \dfrac{((b - \tau)(2z^2 + 1)^2 + (\tau - \sigma))^2}{z^2 (z - i)^3 (z + i)^3} dz.
  \end{equation}
   Then, 
   \begin{align*}
     \Res(\al_1, i) & = 
     -\frac{i}{32}  \theta  (b^2  (128 \sigma^2 - 32 \sigma \tau + 15 \tau ^2 - 111) - 2 b (112 \sigma ^2
     \tau - \sigma  \tau ^2 + \sigma - 112 \tau )\\
     & \hspace{3cm} + 3 \sigma ^2 (37 \tau ^2 - 5)+32 \tau(\sigma  - 4 \tau))
     = -\Res(\al_1, -i ),\\
     \Res(\al_2, i) & = 
     -\frac{1}{32} \theta  (b^2 (128 \sigma ^2 - 32 \sigma \tau + 15 \tau ^2 + 111) - 2 b (112 \sigma ^2 \tau - 
     \sigma  \tau ^2 - \sigma + 112 \tau ))\\
     & \hspace{3cm} + 3 \sigma ^2 (37 \tau ^2 + 5) - 32 \tau(\sigma  + 4 \tau))
      = - \Res(\al_2,- i),\\
      \Res(\al_3, i) & = 
      \frac{i}{16}  \theta  (b^2 (112 \sigma - \tau ) + 2 b (8 \sigma ^2 - 127 \sigma  \tau + 8 \tau ^2)
       - \sigma \tau  (\sigma - 112 \tau ))\\
      &= - \Res(\al_3, -i), \qquad \Res(\al, 0) = \bm 0. 
   \end{align*}
   Hence, the period condition is equivalent to \eqref{eq:per_Dg2_1_(5)}, and 
   one can see that there exist $b, \sigma, \tau$, and $\theta$ that satisfy the condition.
   
  \noindent
  (6)
   The meromorphic $1$-form $\omega$ is expressed as 
  \begin{equation}
    \omega = \theta \dfrac{((b - \tau)(2z^2 + 1)^2 + (\tau - \sigma))^2}{z^3 (z - i)^3 (z + i)^2} dz.
  \end{equation}
  Then, we know that
  \begin{align*}
    32 \Res(\al_1, -i) + 8 \Res(\al_1, 0) &= 3 i \theta (b - \sigma)^2 ( \tau^2 - 1), \\
    32\Res(\al_2, -i) + 8 \Res(\al_2, 0) &= 3 \theta (b - \sigma)^2 (\tau^2 + 1), \\
    - 16 \Res(\al_3, -i) - 4 \Res(\al_3, 0) &= 3 i \theta (b- \sigma)^2 \tau.
  \end{align*}
  The period condition implies that these three values are all in $\R$, 
  which is a contradiction (see \eqref{eq:Dg_2_period_no}).

  \noindent
  (8)
   The meromorphic $1$-form $\omega$ is expressed as 
  \begin{equation}
    \omega = \theta \dfrac{((b - \tau)(2z^2 + 1)^2 + (\tau - \sigma))^2}{z^3 (z - i)^2 (z + i)^2} dz.
  \end{equation}
  Then, 
  \begin{align*}
    \Res(\al_1, i) & = 
    \frac{1}{2} \theta  (b - \sigma ) (b (4 \sigma \tau - \tau^2 - 3) - \sigma (3 \tau^2 + 1) + 4 \tau) \\
    &= \Res(\al_1, - i) = -2 \Res(\al_1, 0), \\
    \Res(\al_2, i) & = 
    -\frac{i}{2}  \theta  (b - \sigma) (b (4 \sigma \tau - \tau^2 + 3) - \sigma(3 \tau^2 - 1)- 4 \tau) \\
    & = \Res(\al_2, - i) = -2 \Res(\al_2, 0), \\
    \Res(\al_3, i) & = 
     -\theta  (b-\sigma ) (b (2 \sigma +\tau ) -\tau  (\sigma + 2 \tau )) 
     = \Res(\al_3, - i) = -2 \Res(\al_3, 0).
  \end{align*}
  Hence, the period condition is equivalent to \eqref{eq:per_Dg2_1_(8)}, and 
  one can see that there exist $b, \sigma, \tau$, and $\theta$ that satisfy the condition.
  
  \noindent
  (9)
   The meromorphic $1$-form $\omega$ is expressed as 
  \begin{equation}
    \omega = \theta \dfrac{((b - \tau)(2z^2 + 1)^2 + (\tau - \sigma))^2}{z^2 (z - i)^3 (z + i)^2} dz.
  \end{equation}
  Then, 
  \begin{align*}
    \Res(\al_1, 0) & = - \fr{1}{2}\theta (b - \sigma)^2 (\tau^2 - 1), 
    \quad 
    \Res(\al_2, 0) = \fr{i}{2} \theta (b - \sigma)^2 (\tau^2 + 1), \\
    \Res(\al_3, 0) & = \theta (b - \sigma)^2 \tau,
  \end{align*}
  which is impossible (see \eqref{eq:Dg_2_period_no}).
  
  Next, we consider {\bf Case 2}.
  By \eqref{eq:G_2-1-1} in Lemma \ref{lem:construction_G}, straightforward computations give that
  $$
    g =  \dfrac{\sigma (b - \tau)(z^2 + 10 i z + 23)^2 +  512 i b (\tau - \sigma)(z - i)}
    {(b - \tau)(z^2  + 10 i z + 23)^2 + 512 i (\tau - \sigma)(z - i)}.
  $$
  
  \noindent
  (1)
  We can set
  \begin{equation}
    \omega = \theta \dfrac{((b - \tau)(z^2  + 10 i z + 23)^2 + 512 i (\tau - \sigma)(z - i))^2}
    				{(z + 17i)^2 (z - i)^2 (z + i)^2} dz.
  \end{equation}
  By
  \begin{align*}
	\Res(\al_1, - i) &= - 64 i \theta (b - \sigma)^2 (\tau^2 - 1), \qquad
	\Res(\al_2,- i) = - 64 \theta (b - \sigma)^2 (\tau ^2 + 1),\\
	\Res(\al_3, - i) &= 128 i \theta \tau (b - \sigma )^2 ,	
  \end{align*}
  \eqref{eq:period_zero} yields a contradiction (see \eqref{eq:Dg_2_period_no}).
  
  Similarly, for the cases (2) through (10), by calculating $\Res(\al, p)$ 
  for $p \in \{\pm i, - 17 i\}$, one can verify that the period condition in each case
  reduces to \eqref{eq:Dg_2_period_no}, which is impossible.
  Thus, {\bf Case 2} does not occur.

    Lastly, let us discuss {\bf Case 4}.
   Taking symmetries $z \mapsto - z$ and $z \mapsto 1 / z$ into consideration 
   (if necessary, change roles of $\sigma$ and $\tau$).
   Also, by considering a coordinate change $z \mapsto i (z + i) / (z - i)$ that 
   fixes the set of ends, changing roles of $\sigma$ and $\tau$, and retaking $\theta$, 
   we see that the case (5) can be identified with (8), so we only consider (5).
   Thus, we only have to check the period problem 
   for the cases (1), (3), (5), (6) in \eqref{eq:omega_form}.
   By \eqref{eq:G_4-2} in Lemma \ref{lem:construction_G}, straightforward computations give that
   $$
     g =  \dfrac{\sigma (b - \tau)(z^2 - 1)^2 + 4 b (\sigma - \tau) z^2}
               	{(b - \tau)(z^2 - 1)^2 + 4 (\sigma - \tau) z^2}.
   $$
   The cases (1), (3), and (6) are impossible since these reduces \eqref{eq:Dg_2_period_no}.
  
  For the case (5), we can set
  \begin{equation}
     \omega = \theta \dfrac{((b - \tau)(z^2 - 1)^2 + 4 (\sigma - \tau) z^2)^2}
    				{z^2 (z - i)^3 (z + i)^3} dz,
  \end{equation}
  and obtain
  \begin{align*}
	\Res(\al_1, i) & = 
	-\frac{i}{2}  \theta  (b - \sigma ) (b (4 \sigma\tau - \tau^2 - 3) - \sigma (3 \tau^2 + 1) + 4 \tau )
	 = - \Res(\al_1, -i),\\
	\Res(\al_2, i) & =  
	-\frac{1}{2} \theta (b - \sigma) (b (4 \sigma \tau - \tau^2 + 3) - \sigma( 3\tau^2 - 1) - 4 \tau ) 
	 = - \Res(\al_2, -i),\\
	\Res(\al_3, i) & = i \theta (b - \sigma) (b (2 \sigma + \tau ) - \tau (\sigma + 2 \tau ))
	 = - \Res(\al_3, -i), \quad
	\Res(\al, 0) = \bm 0.
  \end{align*}
  In this case, One can see that there exist $b, \sigma, \tau$, and $\theta$ that satisfy the 
  period condition.

  \end{proof}

  \begin{figure}[h]
  \centering 
  \begin{tabular}{c@{\hspace{1cm}}c}
    \includegraphics[width = 55mm]{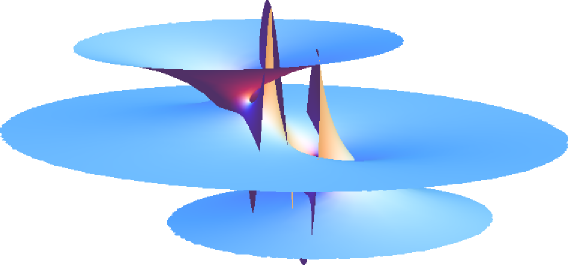} &
    \includegraphics[width = 50mm]{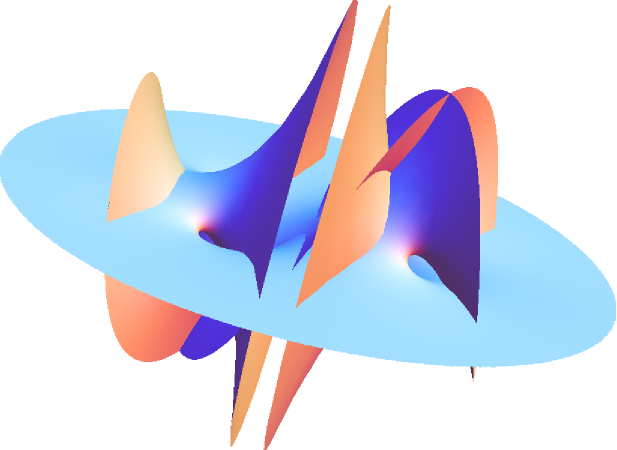} \\
    {\small (1)} & {\small (2)} \\
    \includegraphics[width = 55mm]{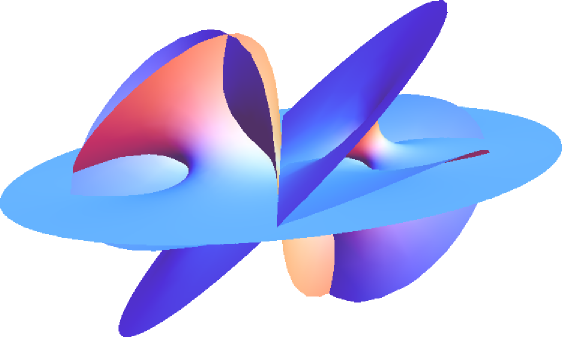} &
    \\   
     {\small (5)} & 
  \end{tabular}
  \caption{
    Complete minimal surface with $C(\Sigma) = - 16 \pi, D_g = 2$ and $\nu_g = 2.5$
    in {\bf Case 1} for each $\omega$ of (1), (2), and (5).
  }
  \label{fig:minimal_surfaces}
\end{figure}
  
  
  \begin{remark}
  We give a few remarks about Theorem \ref{thm:TC16Pi_25}.
  \begin{itemize}
  \item
    The family in (2) of {\bf Case 1} is a generalization of Kawakami--Watanabe's one. 
    In fact, by
    replacing $\tau$ with $\sigma$, $\sigma$ with $\sigma a\ (a \in \R \setminus \{1\}),$
    $b$ with $\sigma b \ (b \in \R \setminus \{a, 1\})$ to satisfy \eqref{eq:KW_period}, 
    and changing coordinate as $z \mapsto 1/ z$, 
    we obtain the family of minimal immersions in Example \ref{ex:KW}.
  \item
    In (8) of {\bf Case 1}, setting $w = z^2$, and in (5) of {\bf Case 4}, setting $w = z - 1 / z$, 
    respectively,
    we see that they are double coverings of complete minimal surfaces 
    of total curvature $- 8 \pi$ that have three catenoid ends, 
    which are given by L\' opez \cite[Theorem 3]{Lop92}.
    
  \end{itemize}
    
  \end{remark}
  
  \subsection{The case where $D_g \leq 1$}

 First, Corollary \ref{cor:surj_25} directly states the following assertion.
 \begin{proposition} \label{prop_16pi_Dg=0}
 There are no complete minimal immersions from $\Chat \setminus \{\text{$4$ points}\} $
 into $\R^3$ of total curvature $-16 \pi$ which satisfy $D_g = 0$ 
 (i.e., the Gauss map $g$ is surjective) and $\nu_g = 2.5$. 
 \end{proposition}

 
 Let us consider the case where $D_g = 1$.
 The definition of $\nu_g$ gives 
 $$
 	\nu_g = 1 + \sum_{j = 1}^{R_g} \left (1 -\fr{1}{\nu_j}\right).
 $$ 
 If $R_g = 1$, then 
 $
 	\nu_g = 2 - 1 / \nu_1 < 2.5,
 $ 
 which is also a contradiction.
 Hence, by \eqref{eq:R_g3}, we know that $R_g = 2$ or $3$.
 
 We now assume $R_g = 2$. 
 Let $\mathcal{TR}_g = \{\sigma, b_1, b_2\}$ be the set of totally ramified values, where 
 $\sigma \in \Chat$ is the unique omitted value and $b_1, b_2 \in \Chat \setminus \{\sigma\}\ (b_1 \neq b_2)$
 are totally ramified values other than the omitted value.
 Then, $\nu_g = 2.5$ and \eqref{eq:nu2ijo} give $(\nu_1, \nu_2) = (4, 4)$.
 Let $\Phi_1 : \Chat \to \Chat$ be a M\"{o}bius transformation defined by 
 $$
 	\Phi_1(w) \coloneqq \fr{\sigma - b_2}{\sigma - b_1} \cdot \fr{w - b_1}{w - b_2},
 $$ 
 which satisfies $\Phi_1(\sigma) = 1, \Phi_1(b_1) = 0, \Phi_1(b_2) = \infty$,
 and let $G \coloneqq \Phi_1 \circ g$.
 Since $(\nu_1, \nu_2) = (4, 4)$ and $\deg(G) = 4$, 
 let $q_1, q_2 \in \Chat \setminus\{\infty, \pm i, t\}\ (q_1 \neq q_2)$ satisfy
 $$
 	g^{-1}(b_1) = G^{-1}(0) = \{q_1\}, \quad g^{-1}(b_2) = G^{-1}(\infty) = \{q_2\}, \quad \ord_{q_1} G = - \ord_{q_2} G = 4.
 $$
 Since $B_G = 6$, $\deg(G) = 4$, and the value $1$ is the unique omitted value of $G$, 
 we see that $G^{-1}(1) = \{\infty, \pm i, t\}$,
 and $G$ does not branch on $\Chat$ except $q_1$ and $q_2$.
 Thus, we can set 
 $$
 	G = \left(\fr{z - q_1}{z - q_2} \right)^4.
 $$ 
 
 Let us determine $q_1, q_2$ and $t$ from the condition $G(\pm i) = G(t) = 1$.
 Let $N(z) \coloneqq  - 4 (q_1 - q_2) z^3 + 6 (q_1^2 - q_2^2) z^2 - 4(q_1^3 - q_2^3) z +q_1^4 - q_2^4$ 
 be the numerator of $G - 1$.
 Since the polynomial $N$ has $\pm i$ and $t$ as the simple roots, Vieta's formula gives  
 \begin{align*}
 	i + (-i) + t &= t = \fr{3}{2} (q_1 + q_2), \qquad
 	i \cdot (-i) + (-i) \cdot t + t \cdot i = 1 = q_1^2 + q_1 q_2 + q_2^2,\\
 	i \cdot (-i) \cdot t &= t = \fr{1}{4} (q_1 + q_2) (q_1^2 + q_2^2).
 \end{align*}
 
 Assume that $q_1 + q_2 = 0$. 
 Then, we have 
 $$
   t = 0,
 $$ 
 and hence $q_1 = \pm 1$ and $q_2 = \mp 1$, double sign in same order.
 As a result, without loss of generality, we get 
 \begin{equation} \label{eq_Dg1_NO1}
 	G = \left(\fr{z - 1}{z + 1} \right)^4.
 \end{equation}
 On the other hand, if $q_1 + q_2 \neq 0$, then by direct computations, we have 
 (1) $t = 3 i$ and $\{q_1, q_2\} = \{i + 2, i - 2\}$ or (2) $t = - 3 i$ and $\{q_1, q_2\} = \{- i + 2, - i - 2\}$.
 In both cases, by coordinate change $z \mapsto (z + 3 i)/(i z + 1)$ and $z \mapsto -z$ in (1), and
 by $z \mapsto (- i z + 3)/(-z + i)$ and $z \mapsto -z$ in (2) respectively, we find that
 the constructed $G$ coincides with \eqref{eq_Dg1_NO1}. 
 
    By $g = G \circ \Phi_1^{-1}$, we can determine the form of $g$ as 
    \begin{equation}\label{eq:Gauss_Dg1}
    	g = \fr{b_2 (b_1 - \sigma) (z - 1)^4 + b_1 (\sigma - b_2) (z + 1)^4}
    		{(b_2 - \sigma)(z - 1)^4 + (b_1 - \sigma)(z + 1)^4}. 
    \end{equation}
    Thus, let us consider whether or not each $(g, \omega)$ on $\Chat \setminus \{\infty, \pm i, 0\}$ 
    satisfies the period condition \eqref{eq:period_zero} regarding $g$ in \eqref{eq:Gauss_Dg1}
     and $\omega$ in each case $(1)$ through $(10)$ in \eqref{eq:omega_form}.
    
     Taking symmetries $z \mapsto - z$ and $z \mapsto \pm 1 / z$ into consideration 
   (if necessary, change roles of $b_1$ and $b_2$), 
   we only have to check the period problem 
   for the cases (1), (3), (5), (6), and (8) in \eqref{eq:omega_form}.
   By considering a coordinate change $z \mapsto i (z + i) / (z - i)$ that 
   fixes the set of ends, changing roles of $b_1$ and $b_2$, and retaking $\theta$, 
   we see that the case (5) can be identified with (8), so we only consider (5). 
   

   \noindent
   (1) 
   One can set
   $$
      \omega = \theta \dfrac{((b_2 - \sigma)(z - 1)^4 + (b_1 - \sigma)(z + 1)^4)^2}{z^2 (z - i)^2 (z + i)^2}dz,
   $$
   The computations of residues of $\al$ at the end $z = \pm i$ yield
       $
         \Res(\al_1, i) - \Res(\al_1, -i) = 4 i \theta (b_1 - b_2)^2  (\sigma^2 - 1), \,
         \Res(\al_2, i) - \Res(\al_2, -i) = 4 \theta (b_1 - b_2)^2  (\sigma^2 + 1),\, 
	$
	and 
	$
         \Res(\al_3, i) - \Res(\al_3, -i) = -8 i \theta (b_1 - b_2)^2  \sigma,
       $
       and
       the period condition \eqref{eq:period_zero} gives
       \begin{equation} \label{eq:Dg_1_period}
          (\sigma^2 - 1) U \in i\R, \qquad 
          (\sigma^2 + 1) U \in \R, \qquad 
          \sigma U \in i \R,
       \end{equation}
       which give a contradiction (see \eqref{eq:Dg_2_period_no}).
   
   \noindent
   (3)
   One can set
   $$
      \omega = \theta \dfrac{((b_2 - \sigma)(z - 1)^4 + (b_1 - \sigma)(z + 1)^4)^2}{z^2 (z - i)^4 (z + i)^2}dz,
   $$
   Straightforward computations give
   \begin{align*}
     &4 \left(
       \fr{8}{5}\Res(\al_1, i) - \fr{2}{5} \Res(\al_2, i) + 4 \Res(\al_1, 0) - \Res(\al_2, 0)
     \right)\\
     & \qquad 
     +\left(
       \fr{2}{5}\Res(\al_1, i) + \fr{8}{5} \Res(\al_2, i) +  \Res(\al_1, 0) + 4 \Res(\al_2, 0)
     \right)
     = - 17 i(\sigma^2 - 1)U,\\
     &\left(
       \fr{8}{5}\Res(\al_1, i) - \fr{2}{5} \Res(\al_2, i) + 4 \Res(\al_1, 0) - \Res(\al_2, 0)
     \right)\\
     & \qquad 
     - 4 \left(
       \fr{2}{5}\Res(\al_1, i) + \fr{8}{5} \Res(\al_2, i) +  \Res(\al_1, 0) + 4 \Res(\al_2, 0)
     \right)
     = 17 (\sigma^2 + 1) U, \\
     &2 \Res(\al_3, i) + 5 \Res(\al_3, 0) = 10 i \sigma U,
   \end{align*}
   where $U \coloneqq \theta (b_1 - b_2)$,
   and the period condition leads to
   a contradiction (see \eqref{eq:Dg_2_period_no}).
   
   \noindent
   (5)
   One can set
   $$
      \omega = \theta \dfrac{((b_2 - \sigma)(z - 1)^4 + (b_1 - \sigma)(z + 1)^4)^2}{z^2 (z - i)^3 (z + i)^3}dz,
   $$
   One can compute residues of $\al$ at each end as 
      \begin{align*}
      	\Res(\al_1, i) 
      	& =
      	\fr{i}{2}\theta \left(
      	  32 \sigma (b_2 - \sigma) + 13 b_2^2 (\sigma^2 - 1) + b_1^2 (32 b_2^2 - 32 b_2 \sigma +13 (\sigma^2 - 1))
      	  \right.\\
      	  &\left.
      	  \hspace{3cm} + b_1 (32 \sigma - 32 b_2^2 \sigma + 6 b_2 (\sigma^2 - 1))
      	\right) = - \Res(\al_1, -i),\\
      	\Res(\al_2, i) 
      	& = 
      	- \fr{1}{2}\theta \left(
      	  32 \sigma (b_2 - \sigma) - 13 b_2^2 (\sigma^2 + 1) - b_1^2 (32 b_2^2 - 32 b_2 \sigma +13 (\sigma^2 + 1))
      	  \right.\\
      	  &\left.
      	  \hspace{3cm} + b_1 ( 32 \sigma + 32 b_2^2 \sigma - 6 b_2 (\sigma^2 + 1))
      	\right) = \Res(\al_2, -i),\\
      	\Res(\al_3, i)
      	&=
      	 - i \theta (b_1^2 (16 b_2 - 3\sigma) - b_2 \sigma(3b_2 - 16 \sigma) + 2b_1 (8 b_2^2 - 29 b_2 \sigma + 8 \sigma^2)) \\
      	 & = - \Res(\al_3,-i), \\
      	  \Res(\al_1, 0) 
      	  &= 
      	  4 \theta  (b_1 - b_2) (b_1 (2 b_2 \sigma -\sigma ^2 - 1 ) - b_2 (\sigma^2 + 1) + 2 \sigma ), \\
      	  \Res(\al_2, 0)
      	  &=
      	  -4 i \theta  (b_1 - b_2) (b_1(2 b_2 \sigma - \sigma ^2 + 1) -  b_2( \sigma^2 - 1) - 2 \sigma ), \\
      	  \Res(\al_3, 0)
      	  & = 
      	  -8 \theta  (b_1 - b_2) (b_1 b_2 - \sigma^2).
      \end{align*}
      One can check that there exist the distinct number $\sigma \in \C, b_1, b_2 \in \Chat,$ 
      and the number $\theta \in \C \setminus \{0\}$ satisfying the period condition
      (see Proposition \ref{prop_16pi_Dg=1}).
   
   \noindent
   (6)
   One can set
   $$
      \omega = \theta \dfrac{((b_2 - \sigma)(z - 1)^4 + (b_1 - \sigma)(z + 1)^4)^2}{z^3 (z - i)^3 (z + i)^2}dz,
   $$
   Then, one can see that 
   \begin{align*}
     &\Res(\al_1,i) + 2 \Res(\al_4, -i) + \Res(\al_7, 0)  = - i (\sigma^2 - 1)U, \\
     &\Res(\al_2, i) + 2 \Res(\al_5, -i) + \Res(\al_8, 0)  = -  (\sigma^2 + 1)U, \\
     &\Res(\al_3, i) + 2 \Res(\al_3, -i) + \Res(\al_3, 0) = 2 i \sigma U,
   \end{align*}
   where $U = \theta(b_1 - b_2)$, 
   and the period condition leads to a contradiction (see \eqref{eq:Dg_2_period_no}).
   Summarizing the arguments above, we obtain the following.

 \begin{proposition} \label{prop_16pi_Dg=1}
    A complete minimal immersion 
    $\Sigma \coloneqq \Chat \setminus \{\text{$4$ points}\} \to \R^3$ 
    of total curvature $-16 \pi$ with $D_g = 1, R_g = 2$ and $\nu_ g = 2.5$
    is given by the following Weierstrass data $(g, \omega)$ defined on $\Sigma = \Chat \setminus \{\infty, \pm i, 0\}$ :
    for distinct values $\sigma \in \C, b_1, b_2 \in \Chat$, 
    and the values $\theta \in \C \setminus\{0\}$,
    \begin{equation}
        g = \dfrac{b_2 (b_1 - \sigma) (z - 1)^4 + b_1 (\sigma - b_2) (z + 1)^4}
    		{(b_1 - \sigma)(z - 1)^4 +  (\sigma - b_2)(z + 1)^4},
    \end{equation}
    and the holomorphic $1$-form $\omega$ meromorphic on $\Chat$ 
    is given by 
         \begin{equation}
      	    (5) \qquad
      	    \omega = \theta \dfrac{((b_2 - \sigma)(z - 1)^4 + (b_1 - \sigma)(z + 1)^4)^2}{z^2 (z - i)^3 (z + i)^3}dz,
   	 \end{equation}
   	 with
     \begin{equation}
      \left\{
      \begin{array}{l}
      	 \Re\left(
      	 \theta (
      	  32 \sigma (b_2 - \sigma) + 13 b_2^2 (\sigma^2 - 1) + b_1^2 (32 b_2^2 - 32 b_2 \sigma +13 (\sigma^2 - 1))
      	  \right.\\
      	  \left.
      	  \hspace{5cm} + b_1 (32 \sigma - 32 b_2^2 \sigma + 6 b_2 (\sigma^2 - 1))
      	\right) = 0,\\
      	  \Im \left(
      	  \theta(
      	  32 \sigma (b_2 - \sigma) - 13 b_2^2 (\sigma^2 + 1) - b_1^2 (32 b_2^2 - 32 b_2 \sigma +13 (\sigma^2 + 1))
      	  \right.\\
      	  \left.
      	  \hspace{5cm} + b_1 ( 32 \sigma + 32 b_2^2 \sigma - 6 b_2 (\sigma^2 + 1))
      	\right) = 0,\\
      	  \Re\left(
      	  \theta (b_1^2 (16 b_2 - 3 \sigma) - b_2 \sigma(3b_2 - 16 \sigma) + 2b_1 (8 b_2^2 - 29 b_2 \sigma + 8 \sigma^2))
      	  \right) = 0, \\
      	  \Im\left(
      	  \theta  (b_1 - b_2) (b_1 (2 b_2 \sigma -\sigma ^2 - 1 ) - b_2 (\sigma^2 + 1) + 2 \sigma )
      	  \right) = 0, \\
      	  \Re\left(
      	  \theta  (b_1 - b_2) (b_1(2 b_2 \sigma - \sigma ^2 + 1) -  b_2( \sigma^2 - 1) - 2 \sigma )
      	  \right) = 0, \\
      	  \Im\left(
      	   \theta  (b_1 - b_2) (b_1 b_2 - \sigma^2)
      	   \right) = 0
        \end{array}
        \right.
       \end{equation}
       {\nf(}e.g. $\sigma = i, \ b_1 = - 4\rt{10} / 13 + 3 i / 13 , \ b_2 = - \bar{b_1}, \ \theta = 1${\nf)}
       {\nf(}Figure \ref{TRV25_1(5)_fig}{\nf)}.

 \end{proposition}
  
  \begin{figure}[h]
  \begin{center}
   \includegraphics[width = 50mm]{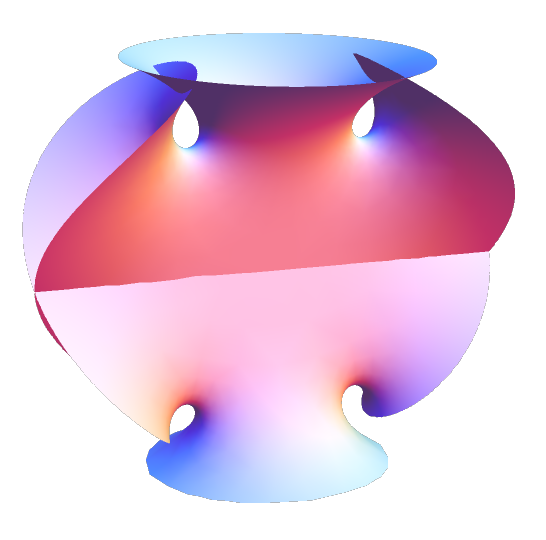}
   \end{center}
   \caption{
   Complete minimal surface with $C(\Sigma) = - 16 \pi, D_g = 1, R_g = 2$ and $\nu_g = 2.5$
   } \label{TRV25_1(5)_fig}
  \end{figure}
 
  \noindent 
  We do not know whether or not there is a complete minimal immersion
 $\Sigma = \Chat \setminus \{\text{$4$ points}\} \to\R^3$ with 
 $C(\Sigma) = -16 \pi, D_g = 1, R_g = 3,$ and $\nu_g = 2.5$.
 One can observe there is a meromorphic function $g : \Sigma \to \Chat$ satisfying such conditions, but
 do not know how to solve the period problem. 
  
  In the final part of this paper, 
  we suggest some open problems for the total weight $\nu_g$ of totally ramified values 
  the Gauss map $g$ of a complete minimal surface in $\R^3$ of finite total curvature.
  \begin{problem}\label{pro:open_TRV}
    \begin{enumerate}[label = (\arabic*), leftmargin = *, font = \normalfont]
      \item{\cite[Section 1]{KKM08}}
        Does there exist some $\kappa \in \R$ satisfying $2.5 \leq \kappa < 4$
        that is an upper bound for $\nu_g$?
      \item
        Does there exist a complete minimal immersion into $\R^3$ of finite total curvature satisfying $\nu_g > 2$, 
        which is of positive genus?
      \item
        What is a necessary and sufficient condition which attains the equality in \eqref{eq:TRBN_est}?
     \item
       What is the best upper bound of the number $D_g + R_g$ of totally ramified values
       for complete minimal immersions into $\R^3$ of finite total curvature?
    \end{enumerate}
  \end{problem}

\newcommand{\etalchar}[1]{$^{#1}$}

\end{document}